\documentclass[journal]{IEEEtran}
\usepackage{stmaryrd}
\usepackage{amsfonts}
\usepackage{graphicx,times,amsmath}
\usepackage{graphics,color}
\usepackage{graphicx}
\usepackage{latexsym}
\usepackage{amsmath}
\usepackage{amsfonts}
\usepackage{amssymb}
\usepackage{url}
\usepackage{subfigure}
\usepackage{times}
\usepackage{color}
\usepackage{ulem}
\newtheorem{theorem}{Theorem}
\newtheorem{corollary}{Corollary}
\newtheorem{definition}{Definition}
\newtheorem{lemma}{Lemma}
\newtheorem{remark}{Remark}
\newtheorem{notation}{Notation}

\allowdisplaybreaks

\graphicspath{{figures/}}

\begin{document}

\title{Global Exponential Stability for Complex-Valued Recurrent Neural Networks With Asynchronous Time Delays}

\author{Xiwei~Liu,~{\it{Member,~IEEE,}}~
        and~Tianping~Chen,~{\it{Senior~Member,~IEEE}}
\thanks{Xiwei Liu is with Department of Computer Science and
Technology, Tongji University, and with the Key Laboratory of Embedded System and Service Computing,
Ministry of Education, Shanghai 200092, P.~R. China. E-mail:
xwliu@tongji.edu.cn}
\thanks{Corresponding Author Tianping Chen is with the School of Computer Sciences/Mathematical
Sciences, Fudan University, 200433, Shanghai, P.~R. China. E-mail:
tchen@fudan.edu.cn}% <-this % stops a space
\thanks{This work was supported by
the National Science Foundation of China under Grant No. 61203149, 61273211, 61233016, the National Basic Research Program of China (973 Program) under
Grant No. 2010CB328101, ``Chen Guang'' project supported by Shanghai
Municipal Education Commission and Shanghai Education Development
Foundation under Grant No. 11CG22, the Fundamental Research Funds
for the Central Universities, and the Program
for Young Excellent Talents in Tongji University.}
}

% make the title area
\maketitle

\begin{abstract}
\boldmath {\bf In this paper, we investigate the global exponential
stability for complex-valued recurrent neural networks with asynchronous
time delays by decomposing complex-valued networks to real and imaginary
parts and construct an equivalent real-valued system. The network model is described by a continuous-time equation. There are two main
differences of this paper with previous works: (1), time delays can be
asynchronous, i.e., delays between different nodes are different, which makes our model more general; (2), we prove the
exponential convergence directly, while the existence and uniqueness of the
equilibrium point is just a direct consequence of the exponential
convergence. By using three generalized norms, we present some sufficient
conditions for the uniqueness and global exponential stability of the
equilibrium point for delayed complex-valued neural networks. These
conditions in our results are less restrictive because of our consideration
of the excitatory and inhibitory effects between neurons, so previous works
of other researchers can be extended. Finally, some numerical simulations are given
to demonstrate the correctness of our obtained results.}
\end{abstract}


\section{Introduction}
Recurrently connected neural networks, including Hopfield neural
networks\cite{H1982}, Cohen-Grossberg neural networks\cite{CG1983}, and
cellular neural networks \cite{CYa}-\cite{CYb}, have been extensively
studied in past decades and found many applications in different areas,
such as signal and image processing, pattern recognition, optimization
problems, associative memories, and so on. Until now, many criteria about
the stability of equilibrium are obtained in the literature, see
\cite{GH1994}-\cite{ZWL2014} and references therein.

It is natural to generalize the real-valued systems to complex-valued systems \cite{MMZ1996}, which can be used in the nonlinear quantum systems, reaction-advection-diffusion systems, heat equation, petri nets, chaotic systems, etc. Many approaches are also obtained, for example, decomposing complex-valued system to two real-valued systems is applied in some nonlinear systems and regular networks, see \cite{MMge2012}-\cite{ZWL2014b} and references therein. Recently, as an important part of nonlinear complex-valued systems, complex-valued neural network (CVNN) models are proposed as an important part of complex-valued systems, and have
attracted more and more attention from various areas in science and
technology, see \cite{JLZ1996}-\cite{FS2014} and references therein. CVNN
can be regarded as an extension of real-valued recurrent neural networks,
which has complex-valued state, output, connection weight, and activation
functions. For example, they are suited to deal with complex state composed
of amplitude and phase. This is one of the core concepts in physical
systems dealing with electromagnetic, light, ultrasonic, quantum waves, and
so on. Moreover, many applications heavily depend on the dynamical
behaviors of networks. Therefore, analysis of these dynamical behaviors is
a necessary step toward practical design of these neural networks. In
\cite{BRS2011}, a CVNN model on time scales is studied based on delta
differential operator. In \cite{{ZZ2009}}-\cite{TM2012}, discrete-time
CVNNs are also discussed. Stability of complex-valued impulsive system is
investigated by \cite{FS2013}. Until now, there have been various methods
to study the stability of CVNNs, such as the Lyapunov functional method
\cite{OIO2011}, the synthesis method \cite{LFL2009}, and so on.

In particular, in hardware implementation, time delays inevitably occur due
to the finite switching speed of the amplifiers and communication time.
What's more, to process moving images, one must introduce time delays in
the signals transmitted among the cells.
Furthermore, time delay is
frequently a source of oscillation and instability in neural networks.
Therefore, neural networks with time delays have much more complicated
dynamics due to the incorporation of delays, and the stability of delayed
neural networks has become a hot topic of great theoretical and practical
importance, and a great deal of significant results have been reported in
the literature. For example, \cite{ZWLL2014} investigates the stability and
synchronization for discrete-time CVNNs with time-varying delays;
\cite{FS2014b} studies the stability of complex-valued impulsive system
with delay. The global exponential and asymptotical stability of CVNNs with
time-delays is studied by \cite{HW2012} with two assumptions of activation
functions, while \cite{ZLC2014} and \cite{FS2014} point out the mistakes in
the proof of \cite{HW2012} and give some new conditions and  criteria to
ensure the existence, uniqueness, and globally asymptotical stability of
the equilibrium point of CVNNs.

In practice, the interconnections are generally asynchronous, that is to say, the inevitable time delays between different nodes are generally different. For example, in order to model vehicular traffic flow \cite{He2001}-\cite{BI2003}, the reaction delays of drivers should be considered, and for different drivers, the reaction delays are different depending on physical
conditions, drivers' cognitive
and physiological states, etc. Moreover, in the load balancing problem \cite{CTGABHJ2005}, for a computing network consisting of $n$ computers (also called nodes), except for the different communication delays, the task-transfer delays $\tau_{jk}$ also should be considered, which depends on the number of tasks to be transferred from node $k$ to node $j$. More related examples can be found in \cite{SNAMG2011} and references therein. Hence, based on above discussions, it is
necessary to study the dynamical behavior of neural networks with asynchronous time (varying)
delays. To our best knowledge, there have been few works to report the
stability of CVNNs with asynchronous time delays, see \cite{ZLH2014},
\cite{XZS2014}. For example, \cite{ZLH2014} focuses on the existence,
uniqueness and global robust stability of equilibrium point for CVNNs with
multiple time-delays and under parameter uncertainties with respect to two
activation functions; while \cite{XZS2014} investigates the dynamical
behaviors of CVNNs with mixed time delays. However, all these works
(\cite{HW2012}-\cite{FS2014}, \cite{ZLH2014}, \cite{XZS2014}) apply the homeomorphism mapping approach
proposed by \cite{FT1995} to prove the existence, uniqueness and global
stability of equilibrium point by two steps: step 1, prove the existence of
equilibrium; step 2, prove its stability.
%{\color{red}
In \cite{CA2001} and
\cite{C2001}, a direct approach to analyze global and local stability of
networks was first proposed. It was revealed that the finiteness of
trajectory $x(t)$ under some norms, i.e.,
$\int_0^{\infty}\|\dot{x}(t)\|dt<\infty$, is a sufficient condition for the
existence, and global stability of the equilibrium point. This idea was
also used in \cite{Mfa}. In this paper, we will adopt this approach.
Moreover, we give several criteria based on three generalized
$L_{\infty}$ norm, $L_{1}$ norm, $L_{2}$ norm, respectively. In particular,
based on $L_{\infty}$-norm, we can discuss the
networks with time-varying delays.

This paper is organized as follows. In Section \ref{pre}, we give the model
description, decompose the complex-valued differential equations to real
part and imaginary part, and then recast it into an equivalent real-valued
differential system, whose dimension is double that of the original
complex-valued system. Some definitions, lemmas and notations used in the
paper are also given. In Section \ref{main}, we present some criteria for
the uniqueness and global exponential stability of the equilibrium point
for recurrent neural networks models with asynchronous time delays by using
the generalized $\infty$-norm, $1$-norm, and $2$-norm, respectively. Some
comparisons with previous M-matrix results are also presented. In Section
\ref{numerical}, some numerical simulations under constant and time varying-delays are given to demonstrate the
effectiveness of our obtained results. Finally, conclusion is given and some discussions about our future investigation of CVNNs are presented in
Section \ref{conc}.

\section{Preliminaries}\label{pre}
In this section, we give some definitions, lemmas and notations, which will be used throughout the paper.

At first, let us give a definition of asynchronous time delays.
\begin{definition}\label{asy}
(Synchronous and asynchronous time delays) For any node $j$ in a coupled neural network, the synchronous time delay means that at time $t$, node $j$ receives the information from other nodes at the same time $t-\tau_j(t)$; while the asynchronous time delays mean that at time $t$, node $j$ receives the information from other nodes at different times $t-\tau_{jk}(t)$, i.e., for nodes $k_1\ne k_2$, $\tau_{jk_1}(t)$ and $\tau_{jk_2}(t)$ can be different.
\end{definition}

Obviously, the network models of asynchronous time delays have a larger scope than that of synchronous time delays.

In this paper, we will investigate the CVNN with asynchronous time delays as follows:
\begin{align}\label{one}
\dot{z}_j(t)&=-d_jz_j(t)\nonumber\\
&+\sum_{k=1}^na_{jk}f_k(z_k(t))+\sum_{k=1}^nb_{jk}g_k(z_k(t-\tau_{jk}))+u_j,\nonumber\\
&~~~~~~~~~~~~~~~~~~~~~~~~~~~~~~~~~~~~~~~~~~~j=1,\cdots,n
\end{align}
where $z_j\in \mathbb{C}$ is the state of $j$-th neuron, $\mathbb{C}$ is
the set of complex numbers; $d_j>0$ represents the positive rate with which
the $j$-th unit will reset its potential to the resting state in
isolation when disconnected from the network; $f_j(\cdot): \mathbb{C}\rightarrow \mathbb{C}$ and $g_j(\cdot): \mathbb{C}\rightarrow \mathbb{C}$ are complex-valued activation functions; matrices $A=(a_{jk})$ and $B=(b_{jk})$ are complex-valued connection weight matrices without and with time delays; $\tau_{jk}$ are asynchronous constant time delays; $u_j\in \mathbb{C}$ is the $j$-th external input.

\begin{remark}
When $\tau_{jk}=\tau$, system (\ref{one}) becomes the model investigated in
\cite{HW2012}; when activation functions $f_j$ and $g_j$ are real
functions, system (\ref{one}) becomes the model investigated by
\cite{C2001}. Therefore, this model has a larger scope than previous works,
and all the obtained results in the next section can be applied to these
special cases.
\end{remark}

For any complex number $z$, we use $z^R$ and $z^I$ to denote its real and
imaginary part respectively, so $z=z^R+i\cdot z^I$, where $i$ denotes the
imaginary unit, that is $i=\sqrt{-1}$.

Now, we introduce some classes of activation functions.

\begin{definition}\label{impro}
Assume $f_j(z)$ can be decomposed to its real and imaginary part as
$f_j(z)=f_j^R(z^R,z^I)+if_j^I(z^R,z^I)$ where $z=z^R+iz^I$,
$f_j^R(\cdot,\cdot): R^2\rightarrow R$ and $f_j^I(\cdot,\cdot):
R^2\rightarrow R$. Suppose the partial derivatives of $f_j(\cdot,\cdot)$
with respect to $z^R, z^I: \partial{f}_j^R/\partial{z^R},
\partial{f}_j^R/\partial{z^I}, \partial{f}_j^I/\partial{z^R}$, and
$\partial{f}_j^I/\partial{z^I}$ exist. If these partial derivatives are
continuous, positive and bounded, i.e., there exist positive constant
numbers $\lambda_j^{RR}, \lambda_j^{RI}, \lambda_j^{IR}, \lambda_j^{II}$,
such that
\begin{align}\label{h1}
0<\partial{f}_j^R/\partial{z^R}\le \lambda_j^{RR}, &~~ 0<\partial{f}_j^R/\partial{z^I}\le \lambda_j^{RI},\nonumber\\
0<\partial{f}_j^I/\partial{z^R}\le \lambda_j^{IR}, &~~ 0<\partial{f}_j^I/\partial{z^I}\le \lambda_j^{II},
\end{align}
then $f_j(z)$ is said to belong to class $H_1(\lambda_j^{RR}, \lambda_j^{RI}, \lambda_j^{IR}, \lambda_j^{II})$.
\end{definition}
\begin{remark}
If $f_j^R$ and $f^I_j$ are absolutely continuous, then their partial
derivatives exist almost everywhere.
\end{remark}

\begin{definition}\label{lip}
Assume $g_j(z)$ can be decomposed to its real and imaginary part as
$g_j(z)=g_j^R(z^R, z^I)+ig_j^I(z^R, z^I)$, where $z=z^R+iz^I$,
$g_j^R(\cdot,\cdot): R^2\rightarrow R$ and $g_j^I(\cdot,\cdot):
R^2\rightarrow R$. Suppose the partial derivatives of $g_j(\cdot,\cdot)$
with respect to $z^R, z^I: \partial{g}_j^R/\partial{z^R},
\partial{g}_j^R/\partial{z^I}, \partial{g}_j^I/\partial{z^R}$, and
$\partial{g}_j^I/\partial{z^I}$ exist. If these partial derivatives are
continuous and bounded, i.e., there exist positive constant numbers
$\mu_j^{RR}, \mu_j^{RI}, \mu_j^{IR}, \mu_j^{II}$, such that
\begin{align}\label{h2}
|\partial{g}_j^R/\partial{z^R}|\le \mu_j^{RR}, &~~ |\partial{g}_j^R/\partial{z^I}|\le \mu_j^{RI},\nonumber\\
|\partial{g}_j^I/\partial{z^R}|\le \mu_j^{IR}, &~~ |\partial{g}_j^I/\partial{z^I}|\le \mu_j^{II},
\end{align}
then $g_j(z)$ is said to belong to class $H_2(\mu_j^{RR}, \mu_j^{RI}, \mu_j^{IR}, \mu_j^{II})$.
\end{definition}

\begin{remark}
Definition \ref{lip} is the usual assumption for activation functions in
the literature of CVNNs, which can be found in \cite{HW2012},
\cite{ZLH2014}, \cite{XZS2014} and references therein. However, the
activation functions defined in Definition \ref{impro} is more restrictive,
which will be useful when considering the signs of entries in connection
weights, i.e., there is a trade-off between the assumption on activation
functions and obtained final criteria.
\end{remark}

Therefore, by decomposing CVNN (\ref{one}) to real and imaginary parts, we
can get two equivalent real-valued systems:
\begin{align}\label{realf}
\dot{z}_j^R(t)&=-d_jz_j^R(t)\nonumber\\
&+\sum_{k=1}^na_{jk}^Rf_k^R\bigg(z_k^R(t), z_k^I(t)\bigg)
-\sum_{k=1}^na_{jk}^If_k^I\bigg(z_k^R(t), z_k^I(t)\bigg)\nonumber\\
&+\sum_{k=1}^nb_{jk}^Rg_k^R\bigg(z_k^R(t-\tau_{jk}), z_k^I(t-\tau_{jk})\bigg)\nonumber\\
&-\sum_{k=1}^nb_{jk}^Ig_k^I\bigg(z_k^R(t-\tau_{jk}), z_k^I(t-\tau_{jk})\bigg)+u_j^R,
\end{align}
and
\begin{align}\label{imagf}
\dot{z}_j^I(t)&=-d_jz_j^I(t)\nonumber\\
&+\sum_{k=1}^na_{jk}^Rf_k^I\bigg(z_k^R(t), z_k^I(t)\bigg)
+\sum_{k=1}^na_{jk}^If_k^R\bigg(z_k^R(t), z_k^I(t)\bigg)\nonumber\\
&+\sum_{k=1}^nb_{jk}^Rg_k^I\bigg(z_k^R(t-\tau_{jk}), z_k^I(t-\tau_{jk})\bigg)\nonumber\\
&+\sum_{k=1}^nb_{jk}^Ig_k^R\bigg(z_k^R(t-\tau_{jk}), z_k^I(t-\tau_{jk})\bigg)+u_j^I.
\end{align}

\begin{remark}
The method of decomposing the CVNNs into two real-valued networks makes the network dimension grow two times, which may cause more calculations. However, this expansion of dimension can also bring some benefits. For example, the number (or dimension) of equilibria can be doubled, which enlarges the capacity of neural networks. It is a trade-off.
\end{remark}

The following three generalized norms are used throughout the paper.
\begin{definition} (See \cite{C2001})
For any vector $v(t)\in R^{m\times 1}$,
\begin{enumerate}
  \item $\{\xi,\infty\}$-norm. $\|v(t)\|_{\{\xi,\infty\}}=\max_j|\xi_j^{-1} v_j(t)|$, where $\xi_j>0, j=1,\cdots,m$.
  \item $\{\xi,1\}$-norm. $\|v(t)\|_{\{\xi,1\}}=\sum_{j}|\xi_j v_j(t)|$, where $\xi_j>0, j=1,\cdots,m$.
  \item $\{\xi,2\}$-norm. $\|v(t)\|_{\{\xi,2\}}=\{\sum_{j}\xi_j| v_j(t)|^2\}^{1/2}$, where $\xi_j>0, j=1,\cdots,m$.
\end{enumerate}
\end{definition}

\begin{lemma}\label{m-matrix} (See \cite{C2001})
Let $C=(c_{jk})\in R^{m\times m}$ be a nonsingular matrix with $c_{jk}\le 0, j,k=1,\cdots,m, j\ne k$. Then all the following statements are equivalent.
\begin{enumerate}
  \item $C$ is an M-matrix, i.e., all the successive principal minors of $C$ are equivalent.
  \item $C^T$ is an M-matrix, where $C^T$ is the transpose of $C$.
  \item The real part of all eigenvalues are positive.
  \item There exists a vector $\xi=(\xi_1,\cdots,\xi_m)^T$ with all $\xi_j>0, j=1,\cdots,m$ such that $\xi^TC>0$, or $C\xi>0$.
\end{enumerate}
\end{lemma}

\begin{notation}
For any real scalar $a$, denote $a^{+}=\max\{0,a\}$. For any matrix $C=(c_{jk})\in R^{n\times n}$, denote $|C|=(|c_{jk}|)$. In the following, we denote $n\times n$ matrices $A^R=(a_{jk}^R)$, $A^I=(a_{jk}^I)$, $B^R=(b_{jk}^R)$, $B^I=(b_{jk}^I)$, and \\ $F^{RR}=\mathrm{diag}\{\lambda_1^{RR},\cdots,\lambda_n^{RR}\},$
$F^{RI}=\mathrm{diag}\{\lambda_1^{RI},\cdots,\lambda_n^{RI}\},$
$F^{IR}=\mathrm{diag}\{\lambda_1^{IR},\cdots,\lambda_n^{IR}\},$
$F^{II}=\mathrm{diag}\{\lambda_1^{II},\cdots,\lambda_n^{II}\},$
$G^{RR}=\mathrm{diag}\{\mu_1^{RR},\cdots,\mu_n^{RR}\},$
$G^{RI}=\mathrm{diag}\{\mu_1^{RI},\cdots,\mu_n^{RI}\},$
$G^{IR}=\mathrm{diag}\{\mu_1^{IR},\cdots,\mu_n^{IR}\},$~~
$G^{II}=\mathrm{diag}\{\mu_1^{II},\cdots,\mu_n^{II}\}$.
\end{notation}

\begin{notation}
For any two non-negative functions $f(t), g(t): (-\infty, +\infty)\rightarrow [0,+\infty)$, $f(t)=O(g(t))$ means that for all $t\in R$, there is a positive constant scalar $c$ such that $f(t)\le c\cdot g(t)$. For any symmetric matric $A$, $\lambda_{max}(A)$ means its largest eigenvalue. A $n$-dimensional vector $p=(p_1,\cdots,p_n)^T$ is called a positive vector, if its all elements are positive, i.e.,  $p_i>0, i=1,\cdots,n$.
\end{notation}

\section{Main Results}\label{main}
In this section, we prove some criteria for the uniqueness and global exponential stability of the equilibrium.

\subsection{Criteria with $\{\xi,\infty\}$-norm}
\begin{theorem}\label{thinfty}
For dynamical systems (\ref{realf}) and (\ref{imagf}), suppose the activation function $f_j(z)$ belongs to class $H_1(\lambda_j^{RR}, \lambda_j^{RI}, \lambda_j^{IR}, \lambda_j^{II})$ and $g_j(z)$ belongs to class $H_2(\mu_j^{RR}, \mu_j^{RI}, \mu_j^{IR}, \mu_j^{II})$, $j=1,\cdots,n$. If there exists a positive vector $\xi=(\xi_1,\cdots,\xi_{n},\phi_1,\cdots,\phi_n)^T>0$ and $\epsilon>0$, such that, for $j=1,\cdots,n$,
\begin{align}\label{t1}
&T1(j)=\xi_{j}\bigg(-d_{j}+\epsilon+\{a_{jj}^R\}^{+}\cdot\lambda_{j}^{RR}+\{-a_{jj}^I\}^{+}\cdot\lambda_{j}^{IR}\bigg)
\nonumber\\
&+\sum_{k=1,k\ne j}^n\xi_{k}|a_{jk}^R|\lambda_{k}^{RR}
+\sum_{k=1}^n\phi_{k}|a_{jk}^R|\lambda_k^{RI}+\sum_{k=1,k\ne j}^n\xi_{k}|a_{jk}^I|\lambda_k^{IR}\nonumber\\
&+\sum_{k=1}^n\phi_{k}|a_{jk}^I|\lambda_k^{II}
+\bigg(\sum\limits_{k=1}^n\xi_{k}|b_{jk}^R|\mu_k^{RR}+\sum\limits_{k=1}^n\phi_{k}|b_{jk}^R|\mu_k^{RI}\nonumber\\
&+\sum\limits_{k=1}^n\xi_{k}|b_{jk}^I|\mu_k^{IR}+\sum\limits_{k=1}^n\phi_{k}|b_{jk}^I|\mu_k^{II}\bigg)e^{\epsilon \tau_{jk}}\le 0,
\end{align}
and
\begin{align}\label{t2}
&T2(j)=\phi_{j}\bigg(-d_{j}+\epsilon+\{a_{jj}^R\}^{+}\cdot\lambda_{j}^{II}+\{a_{jj}^I\}^{+}\cdot\lambda_{j}^{RI}\bigg)
\nonumber\\
&+\sum_{k=1}^n\xi_{k}|a_{jk}^R|\lambda_{k}^{IR}+\sum\limits_{k=1,k\ne j}^n\phi_{k}|a_{jk}^R|\lambda_k^{II}
+\sum_{k=1}^n\xi_{k}|a_{jk}^I|\lambda_k^{RR}\nonumber\\
&+\sum\limits_{k=1, k\ne j}^n\phi_{k}|a_{jk}^I|\lambda_k^{RI}
+\bigg(\sum\limits_{k=1}^n\xi_{k}|b_{jk}^R|\mu_k^{IR}+\sum\limits_{k=1}^n\phi_{k}|b_{jk}^R|\mu_k^{II}\nonumber\\
&+\sum\limits_{k=1}^n\xi_{k}|b_{jk}^I|\mu_k^{RR}+\sum\limits_{k=1}^n\phi_{k}|b_{jk}^I|\mu_k^{RI}\bigg)e^{\epsilon \tau_{jk}}\le 0,
\end{align}
then dynamical systems (\ref{realf}) and (\ref{imagf}) have a unique equilibrium $\overline{Z}^R=(\overline{z}_1^R,\cdots, \overline{z}_n^R)^T$ and  $\overline{Z}^I=(\overline{z}_1^I,\cdots,\overline{z}_n^I)^T$, respectively. Moreover, for any solution
\begin{align}\label{z}
Z(t)=(z_1^R(t),\cdots,z^R_n(t), z_1^I(t),\cdots,z_n^I(t))^T,
\end{align}
there hold
\begin{align}
\|\dot{Z}(t)\|_{\{\xi,\infty\}}=O(e^{-\epsilon t}),\label{e1}\\
\|Z(t)-({\overline{Z}^R}^T,{\overline{Z}^I}^T)^T\|_{\{\xi,\infty\}}=O(e^{-\epsilon t}).\label{e2}
\end{align}
\end{theorem}

Its proof can be found in Appendix A.

\begin{corollary}\label{cor1}
For dynamical systems (\ref{realf}) and (\ref{imagf}), suppose the activation function $f_j(z)$ belongs to class $H_1(\lambda_j^{RR}, \lambda_j^{RI}, \lambda_j^{IR}, \lambda_j^{II})$ and $g_j(z)$ belongs to class $H_2(\mu_j^{RR}, \mu_j^{RI}, \mu_j^{IR}, \mu_j^{II})$, $j=1,\cdots,n$. If there exists a positive vector $\xi=(\xi_1,\cdots,\xi_{n},\phi_1,\cdots,\phi_n)^T>0$, such that, for $j=1,\cdots,n$,
\begin{align}\label{t3}
&T3(j)=\xi_{j}\bigg(-d_{j}+\{a_{jj}^R\}^{+}\cdot\lambda_{j}^{RR}+\{-a_{jj}^I\}^{+}\cdot\lambda_{j}^{IR}\bigg)
\nonumber\\
&+\sum_{k=1,k\ne j}^n\xi_{k}|a_{jk}^R|\lambda_{k}^{RR}
+\sum_{k=1}^n\phi_{k}|a_{jk}^R|\lambda_k^{RI}+\sum_{k=1,k\ne j}^n\xi_{k}|a_{jk}^I|\lambda_k^{IR}\nonumber\\
&+\sum_{k=1}^n\phi_{k}|a_{jk}^I|\lambda_k^{II}
+\sum\limits_{k=1}^n\xi_{k}|b_{jk}^R|\mu_k^{RR}+\sum\limits_{k=1}^n\phi_{k}|b_{jk}^R|\mu_k^{RI}\nonumber\\
&+\sum\limits_{k=1}^n\xi_{k}|b_{jk}^I|\mu_k^{IR}+\sum\limits_{k=1}^n\phi_{k}|b_{jk}^I|\mu_k^{II} < 0,
\end{align}
\begin{align}\label{t4}
&T4(j)=\phi_{j}\bigg(-d_{j}+\{a_{jj}^R\}^{+}\cdot\lambda_{j}^{II}+\{a_{jj}^I\}^{+}\cdot\lambda_{j}^{RI}\bigg)
\nonumber\\
&+\sum_{k=1}^n\xi_{k}|a_{jk}^R|\lambda_{k}^{IR}+\sum\limits_{k=1,k\ne j}^n\phi_{k}|a_{jk}^R|\lambda_k^{II}
+\sum_{k=1}^n\xi_{k}|a_{jk}^I|\lambda_k^{RR}\nonumber\\
&+\sum\limits_{k=1, k\ne j}^n\phi_{k}|a_{jk}^I|\lambda_k^{RI}
+\sum\limits_{k=1}^n\xi_{k}|b_{jk}^R|\mu_k^{IR}+\sum\limits_{k=1}^n\phi_{k}|b_{jk}^R|\mu_k^{II}\nonumber\\
&+\sum\limits_{k=1}^n\xi_{k}|b_{jk}^I|\mu_k^{RR}+\sum\limits_{k=1}^n\phi_{k}|b_{jk}^I|\mu_k^{RI}< 0,
\end{align}
then any solution of systems (\ref{realf}) and (\ref{imagf}) respectively converges to a unique equilibrium
exponentially.
\end{corollary}

If conditions (\ref{t3}) and (\ref{t4}) hold, then we can find a sufficient small constant $\epsilon>0$, such that inequalities (\ref{t1}) and (\ref{t2}) hold. Therefore, this corollary is a direct consequence of Thm. \ref{thinfty}.

\begin{corollary}\label{cor2}
For dynamical systems (\ref{realf}) and (\ref{imagf}), suppose the activation function $f_j(z)$ belongs to class $H_2(\lambda_j^{RR}, \lambda_j^{RI}, \lambda_j^{IR}, \lambda_j^{II})$ and $g_j(z)$ belongs to class $H_2(\mu_j^{RR}, \mu_j^{RI}, \mu_j^{IR}, \mu_j^{II})$, $j=1,\cdots,n$. If there exists a positive vector $\xi=(\xi_1,\cdots,\xi_{n},\phi_1,\cdots,\phi_n)^T>0$, such that, for $j=1,\cdots,n$,
\begin{align}\label{t5}
&T5(j)=-\xi_{j}d_{j}
+\sum_{k=1}^n\xi_{k}|a_{jk}^R|\lambda_{k}^{RR}+\sum_{k=1}^n\phi_{k}|a_{jk}^R|\lambda_k^{RI}\nonumber\\
&+\sum_{k=1}^n\xi_{k}|a_{jk}^I|\lambda_k^{IR}
+\sum_{k=1}^n\phi_{k}|a_{jk}^I|\lambda_k^{II}
+\sum\limits_{k=1}^n\xi_{k}|b_{jk}^R|\mu_k^{RR}\nonumber\\
&+\sum\limits_{k=1}^n\phi_{k}|b_{jk}^R|\mu_k^{RI}+\sum\limits_{k=1}^n\xi_{k}|b_{jk}^I|\mu_k^{IR}+\sum\limits_{k=1}^n\phi_{k}|b_{jk}^I|\mu_k^{II} < 0,
\end{align}
\begin{align}\label{t6}
&T6(j)=-\phi_{j}d_{j}+\sum_{k=1}^n\xi_{k}|a_{j,k}^R|\lambda_{k}^{IR}
+\sum\limits_{k=1}^n\phi_{k}|a_{jk}^R|\lambda_k^{II}\nonumber\\
&+\sum_{k=1}^n\xi_{k}|a_{jk}^I|\lambda_k^{RR}
+\sum\limits_{k=1}^n\phi_{k}|a_{jk}^I|\lambda_k^{RI}+\sum\limits_{k=1}^n\xi_{k}|b_{jk}^R|\mu_k^{IR}\nonumber\\
&+\sum\limits_{k=1}^n\phi_{k}|b_{jk}^R|\mu_k^{II}
+\sum\limits_{k=1}^n\xi_{k}|b_{jk}^I|\mu_k^{RR}+\sum\limits_{k=1}^n\phi_{k}|b_{jk}^I|\mu_k^{RI} <0,
\end{align}
then any solution of systems (\ref{realf}) and (\ref{imagf}) respectively converges to a unique equilibrium exponentially.
\end{corollary}

This result is a direct consequence of Corollary \ref{cor1}.

\begin{remark}\label{ws}
Theorem 1 can be generalized to the system with time-varying  delays
\begin{align}\label{onea}
\dot{z}_j(t)&=-d_jz_j(t)\nonumber\\
&+\sum_{k=1}^na_{jk}f_k(z_k(t))+\sum_{k=1}^nb_{jk}g_k(z_k(t-\tau_{jk}(t)))+u_j,\nonumber\\
&~~~~~~~~~~~~~~~~~~~~~~~~~~~~~~~~~~~~~~~~~~~j=1,\cdots,n
\end{align}
where $\tau_{jk}(t)$ can be bounded or unbounded. In fact, by Theorem 1, system
(\ref{onea}) has an equilibrium  $\overline{Z}=(\overline{z}_1,\cdots, \overline{z}_n)^T$, and
\begin{align*}
\frac{d(z_j(t)-\overline{z}_j)}{dt}&=-d_j(z_j(t)-\overline{z}_j)
+\sum_{k=1}^na_{jk}(f_k(z_k(t))-f_k(\overline{z}_k))\nonumber\\
&+\sum_{k=1}^nb_{jk}(g_k(z_k(t-\tau_{jk}(t)))-g_k(\overline{z}_k)).
\end{align*}
Replacing $e^{\epsilon t}\dot{z}(t)$ by $e^{\epsilon t}(z(t)-\overline{Z})$
in the proof of Theorem 1 and with the similar approach, we can prove that under the conditions (\ref{t1}) and (\ref{t2}), system (\ref{onea}) has a unique equilibrium, which is globally $\mu$ stable (for the concept of $\mu$ stability first proposed in \cite{CW2007} and details, readers can refer to \cite{CW2007,LLC}).
\end{remark}

\subsection{Criteria with $\{\xi,1\}$-norm}

\begin{theorem}\label{th1}
For dynamical systems (\ref{realf}) and (\ref{imagf}), suppose the activation function $f_j(z)$ belongs to class $H_1(\lambda_j^{RR}, \lambda_j^{RI}, \lambda_j^{IR}, \lambda_j^{II})$ and $g_j(z)$ belongs to class $H_2(\mu_j^{RR}, \mu_j^{RI}, \mu_j^{IR}, \mu_j^{II})$, $j=1,\cdots,n$. If there exists a positive vector $\xi=(\xi_1,\cdots,\xi_{n},\phi_1,\cdots,\phi_n)^T>0$ and $\epsilon>0$, such that, for $k=1,\cdots,n$,
\begin{align*}
%\label{t7}
T7(k)&=\xi_k(-d_k+\epsilon)\nonumber\\
&+\bigg[\xi_ka_{kk}^R+\sum_{j=1,j\ne k}^n\xi_j|a_{jk}^R|+\sum_{j=1}^n\phi_j|a_{jk}^I|\bigg]^{+}\lambda_k^{RR}\nonumber\\
&+\bigg[-\xi_ka_{kk}^I+\sum_{j=1,j\ne k}^n\xi_j|a_{jk}^I|+\sum_{j=1}^n\phi_j|a_{jk}^R|\bigg]^{+}\lambda_k^{IR}\nonumber\\
&+\sum\limits_{j=1}^n\bigg(\xi_j(|b_{jk}^R|\mu_k^{RR}+|b_{jk}^I|\mu_k^{IR})\nonumber\\
&~~~~~~+\phi_j(|b_{jk}^R|\mu_k^{IR}+|b_{jk}^I|\mu_k^{RR})\bigg)e^{\epsilon \tau_{jk}}\le 0,
%\end{align}
%\begin{align}\label{t8}
\\T8(k)&=\phi_k(-d_k+\epsilon)\nonumber\\
&+\bigg[\phi_ka_{kk}^R+\sum_{j=1,j\ne k}^n\phi_j|a_{jk}^R|+\sum_{j=1}^n\xi_j|a_{jk}^I|\bigg]^{+}\lambda_k^{II}\nonumber\\
&+\bigg[\phi_ka_{kk}^I+\sum_{j=1,j\ne k}^n\phi_j|a_{jk}^I|+\sum_{j=1}^n\xi_j|a_{jk}^R|\bigg]^{+}\lambda_k^{RI}\nonumber\\
&+\sum\limits_{j=1}^n\bigg(\xi_j(|b_{jk}^R|\mu_k^{RI}+|b_{jk}^I|\mu_k^{II})\nonumber\\
&~~~~~~+\phi_j(|b_{jk}^R|\mu_k^{II}+|b_{jk}^I|\mu_k^{RI})\bigg)e^{\epsilon \tau_{jk}}\le 0,
\end{align*}
then dynamical systems (\ref{realf}) and (\ref{imagf}) have a unique equilibrium $\overline{Z}^R=(\overline{z}_1^R,\cdots, \overline{z}_n^R)^T$ and  $\overline{Z}^I=(\overline{z}_1^I,\cdots,\overline{z}_n^I)^T$ respectively. Moreover, for any solution $Z(t)$ defined by (\ref{z}), equations (\ref{e1}) and (\ref{e2}) hold, while the norm is $\{\xi, 1\}$-norm.
\end{theorem}

Its proof can be found in Appendix B.

\begin{corollary}\label{cor3}
For dynamical systems (\ref{realf}) and (\ref{imagf}), suppose the activation function $f_j(z)$ belongs to class $H_1(\lambda_j^{RR}, \lambda_j^{RI}, \lambda_j^{IR}, \lambda_j^{II})$ and $g_j(z)$ belongs to class $H_2(\mu_j^{RR}, \mu_j^{RI}, \mu_j^{IR}, \mu_j^{II})$, $j=1,\cdots,n$. If there exists a positive vector $\xi=(\xi_1,\cdots,\xi_{n},\phi_1,\cdots,\phi_n)^T>0$, such that, for $k=1,\cdots,n$,
\begin{align*}
%\label{t9}
&T9(k)=-\xi_kd_k\nonumber\\
&+\bigg[\xi_ka_{kk}^R+\sum_{j=1,j\ne k}^n\xi_j|a_{jk}^R|+\sum_{j=1}^n\phi_j|a_{jk}^I|\bigg]^{+}\lambda_k^{RR}\nonumber\\
&+\bigg[-\xi_ka_{kk}^I+\sum_{j=1,j\ne k}^n\xi_j|a_{jk}^I|+\sum_{j=1}^n\phi_j|a_{jk}^R|\bigg]^{+}\lambda_k^{IR}\nonumber\\
&+\sum\limits_{j=1}^n\xi_j(|b_{jk}^R|\mu_k^{RR}+|b_{jk}^I|\mu_k^{IR})+\phi_j(|b_{jk}^R|\mu_k^{IR}+|b_{jk}^I|\mu_k^{RR})\nonumber\\
&< 0,
%\end{align}
%\begin{align}\label{t10}
\\&T10(k)=-\phi_kd_k\nonumber\\
&+\bigg[\phi_ka_{kk}^R+\sum_{j=1,j\ne k}^n\phi_j|a_{jk}^R|+\sum_{j=1}^n\xi_j|a_{jk}^I|\bigg]^{+}\lambda_k^{II}\nonumber\\
&+\bigg[\phi_ka_{kk}^I+\sum_{j=1,j\ne k}^n\phi_j|a_{jk}^I|+\sum_{j=1}^n\xi_j|a_{jk}^R|\bigg]^{+}\lambda_k^{RI}\nonumber\\
&+\sum\limits_{j=1}^n\xi_j(|b_{jk}^R|\mu_k^{RI}+|b_{jk}^I|\mu_k^{II})+\phi_j(|b_{jk}^R|\mu_k^{II}+|b_{jk}^I|\mu_k^{RI})\nonumber\\
&<0,
\end{align*}
then any solution of systems (\ref{realf}) and (\ref{imagf}) respectively converges to a unique equilibrium
%$\overline{z}^R$ and $\overline{z}^I$
exponentially.
%Its derivative $\dot{Z}(t)$ converges to zero exponentially.
\end{corollary}

\begin{corollary}\label{cor4}
For dynamical systems (\ref{realf}) and (\ref{imagf}), suppose the activation function $f_j(z)$ belongs to class $H_2(\lambda_j^{RR}, \lambda_j^{RI}, \lambda_j^{IR}, \lambda_j^{II})$ and $g_j(z)$ belongs to class $H_2(\mu_j^{RR}, \mu_j^{RI}, \mu_j^{IR}, \mu_j^{II})$, $j=1,\cdots,n$. If there exists a positive vector $\xi=(\xi_1,\cdots,\xi_{n},\phi_1,\cdots,\phi_n)^T>0$, such that, for $k=1,\cdots,n$,
\begin{align}\label{t11}
&T11(k)=-\xi_kd_k+\bigg[\sum_{j=1}^n\xi_j|a_{jk}^R|+\sum_{j=1}^n\phi_j|a_{jk}^I|\bigg]\lambda_k^{RR}
\nonumber\\&+\bigg[\sum_{j=1}^n\xi_j|a_{jk}^I|+\sum_{j=1}^n\phi_j|a_{jk}^R|\bigg]\lambda_k^{IR}\nonumber\\
&+\sum\limits_{j=1}^n\xi_j(|b_{jk}^R|\mu_k^{RR}+|b_{jk}^I|\mu_k^{IR})+\phi_j(|b_{jk}^R|\mu_k^{IR}+|b_{jk}^I|\mu_k^{RR})\nonumber\\
&< 0,
\end{align}
\begin{align}\label{t12}
&T12(k)=-\phi_{k}d_{k}+\bigg[\sum_{j=1}^n\phi_{j}|a_{jk}^R|+\sum_{j=1}^n\xi_j|a_{jk}^I|\bigg]\lambda_{k}^{II}\nonumber\\
&+\bigg[\sum_{j=1}^n\phi_{j}|a_{j,k}^I|+\sum_{j=1}^n\xi_j|a_{jk}^R|\bigg]\lambda_{k}^{RI}\nonumber\\
&+\sum\limits_{j=1}^n\xi_j(|b_{jk}^R|\mu_{k}^{RI}+|b_{jk}^I|\mu_{k}^{II})+\phi_{j}(|b_{jk}^R|\mu_{k}^{II}+|b_{jk}^I|\mu_{k}^{RI})\nonumber\\
&< 0,
\end{align}
then any solution of systems (\ref{realf}) and (\ref{imagf}) respectively converges to a unique equilibrium
%$\overline{z}^R$ and $\overline{z}^I$
exponentially.
%Its derivative $\dot{Z}(t)$ converges to zero exponentially.
\end{corollary}

\subsection{Some comparisons}
The following theorem is a direct consequence of Corollary \ref{cor2}, Corollary \ref{cor4} and the properties of the M-matrix.
\begin{theorem}\label{matrix}
For dynamical systems (\ref{realf}) and (\ref{imagf}), suppose the activation function $f_j(z)$ belongs to class $H_2(\lambda_j^{RR}, \lambda_j^{RI}, \lambda_j^{IR}, \lambda_j^{II})$ and $g_j(z)$ belongs to class $H_2(\mu_j^{RR}, \mu_j^{RI}, \mu_j^{IR}, \mu_j^{II})$, $j=1,\cdots,n$. Denote
\begin{align}\label{notg}
&\overline{D}=\left(\begin{array}{cc}D&0\\0&D\end{array}\right),\nonumber\\
&\overline{A}=\left(\begin{array}{cc}|A^R|&|A^I|\\|A^I|&|A^R|\end{array}\right),
\overline{F}=\left(\begin{array}{cc}F^{RR}&F^{RI}\\ F^{IR}&F^{II}\end{array}\right),\nonumber\\
&\overline{B}=\left(\begin{array}{cc}|B^R|&|B^I|\\|B^I|&|B^R|\end{array}\right),
\overline{G}=\left(\begin{array}{cc}G^{RR}&G^{RI}\\ G^{IR}&G^{II}\end{array}\right).
\end{align}
If $\overline{D}-\overline{A}\overline{F}-\overline{B}\overline{G}$ is a nonsingular M-matrix, then any solution of systems (\ref{realf}) and (\ref{imagf}) respectively converges to a unique equilibrium %$\overline{z}^R$ and $\overline{z}^I$
exponentially.
%Its derivative $\dot{Z}(t)$ converges to zero exponentially.
\end{theorem}
\begin{proof}
If $\overline{D}-\overline{A}\overline{F}-\overline{B}\overline{G}$ is a nonsingular M-matrix, according to Lemma \ref{m-matrix}, there exists vector $\xi=(\xi_1,\cdots,\xi_{n},\phi_1,\cdots,\phi_n)^T>0$, such that $(\overline{D}-\overline{A}\overline{F}-\overline{B}\overline{G})\xi>0$, that is, inequalities (\ref{t5}) and (\ref{t6}) hold. Therefore, the conclusion is a direct consequence of Corollary \ref{cor2}.

On the other hand, if $\overline{D}-\overline{A}\overline{F}-\overline{B}\overline{G}$ is a nonsingular M-matrix, according to Lemma \ref{m-matrix}, there exists vector $\xi=(\xi_1,\cdots,\xi_{n},\phi_1,\cdots,\phi_n)^T>0$, such that $\xi^T(\overline{D}-\overline{A}\overline{F}-\overline{B}\overline{G})>0$, that is, inequalities (\ref{t11}) and (\ref{t12}) hold. Therefore, the conclusion is also a direct consequence of Corollary \ref{cor4}.
\end{proof}

\begin{remark}\label{key}
Criterion based on M-matrix was also reported in
\cite{HW2012}. However, it neglects the signs of entries in the connection
matrices $A$ and $B$, and thus, the difference between excitatory and
inhibitory effects might be ignored. Comparatively, the criteria given in Theorem
\ref{thinfty}, Theorem \ref{th1}, Corollary \ref{cor1}, Corollary \ref{cor3}
are more powerful.

In the following, we give a comparison between Corollary \ref{cor1} and Theorem \ref{matrix} by using the matrix theory. Denote matrices
\begin{align*}
P_1&=\mathrm{diag}(|a_{11}^R|-\{a_{11}^R\}^{+},\cdots,|a_{nn}^R|-\{a_{nn}^R\}^{+});\\
P_2&=\mathrm{diag}(|a_{11}^I|-\{-a_{11}^I\}^{+},\cdots,|a_{nn}^I|-\{-a_{nn}^I\}^{+});\\
P_3&=\mathrm{diag}(|a_{11}^I|-\{a_{11}^I\}^{+},\cdots,|a_{nn}^I|-\{a_{nn}^I\}^{+}).
\end{align*}
Obviously, these matrices are all non-negative definite. Define
\begin{align}\label{del}
\overline{\Delta}=\left(\begin{array}{cc}
P_1F^{RR}+P_2F^{IR}&0\\0&P_1F^{II}+P_3F^{RI}
\end{array}\right),
\end{align}
so it is also non-negative definite. Using this notation, and from Corollary \ref{cor1}, the sufficient condition for global stability is that
\begin{align}\label{old}
\overline{D}-\overline{A}\overline{F}-\overline{B}\overline{G}+\overline{\Delta}
\end{align}
should be a nonsingular M-matrix. Obviously, if $\overline{D}-\overline{A}\overline{F}-\overline{B}\overline{G}$ is a nonsingular M-matrix, the above matrix (\ref{old}) is also a nonsingular M-matrix; instead, if matrix (\ref{old}) is a nonsingular M-matrix, $\overline{D}-\overline{A}\overline{F}-\overline{B}\overline{G}$ may be not.

Therefore, Corollary \ref{cor1} presents a better criterion than that by previous works, like \cite{HW2012}, because it considers the signs of entries in the connection
matrix $A$, whose positive effect is described by the above nonnegative matrix $\overline{\Delta}$ defined in (\ref{del}). Moreover, from this result, we can also find that in order to make the CVNNs have the stable equilibrium, $P_1, P_2, P_3$ should be as large as possible, so one way is to make all $a_{jj}^R, j=1,\cdots,n$ be negative numbers.
\end{remark}

\begin{remark}
The function $M(t)=\max_{t}\max_{i=1,\cdots,m}|u_i(t)|$
proposed in \cite{C2001} is a powerful tool in dealing with delayed
systems. In particular, for the time-varying delays.
\end{remark}

\begin{remark} It can be seen that in computing the integral $\int_{0}^{\infty}||\dot{Z}(t)||dt$, the estimation of $\frac{d}{dt}||\dot{Z}(t)||$ plays an important
role.

Let $A(t)=(a_{ij})_{i,j=1}^{N}$,
$\xi_{i}>0$, $i=1,\cdots,N$ and
\begin{equation}
\frac{dw}{dt}=Aw(t) \label{Lemma2}
\end{equation}
It has been shown that (see \cite{CA2001,C2001})
\begin{align*}
\max\frac{\frac{d}{dt}\|w(t)\|_{\{\xi, 1\}}}{\|w(t)\|_{\{\xi,
1\}}}&=\max_{j}[a_{jj}+\sum_{i\ne
j}\frac{\xi_{i}}{\xi_{j}}|a_{ij}|],
\\
\max\frac{\frac{d}{dt}\|w(t)\|_{\{\xi, \infty\}}}{\|w(t)\|_{\{\xi,
\infty\}}}&=\max_{i}[a_{ii}+\sum_{j\ne
i}\frac{\xi_{j}}{\xi_{i}}|a_{ij}|],\\
\max\frac{\frac{d}{dt}\|w(t)\|_{\{\xi, 2\}}^2}{\|w(t)\|_{\{\xi,
2\}}^2}&=\lambda_{max}(\Xi A+A^T\Xi), ~\Xi=\mathrm{diag}(\xi).
\end{align*}
\begin{align*}
&\frac{d}{dt}\|w(t)\|_{\{\xi, 1\}}
=\sum_{i=1}^n sign(w_{i}(t))\xi_{i} \sum_{j=1}^n a_{ij}w_{j}(t)\\ &=
\sum_{j=1}^n [\sum_{i}sign(w_{i}(t))\frac{\xi_{i}}{\xi_j}a_{ij}]\xi_jw_{j}(t)\\ &\le  \sum_{j}[a_{jj}+\sum_{i\ne
j}\frac{\xi_{i}}{\xi_{j}}|a_{ij}|]|\xi_{j}w_{j}(t)|\\
\nonumber &\le  \max_{j}[a_{jj}+\sum_{i\ne
j}\frac{\xi_{i}}{\xi_{j}}|a_{ij}|]\|w(t)\|_{\{\xi, 1\}}
\end{align*}
Therefore,
\begin{align*}
\max\frac{\frac{d}{dt}\|w(t)\|_{\{\xi, 1\}}}{\|w(t)\|_{\{\xi,
1\}}}&=\max_{j}[a_{jj}+\sum_{i\ne
j}\frac{\xi_{i}}{\xi_{j}}|a_{ij}|].
\end{align*}
Similarly, we can prove the other two equalities.

These three equalities play very important role in discussing stability of the neural networks or other dynamical systems.
For example, if $\max_{j}[a_{jj}+\sum_{i\ne
j}\frac{\xi_{i}}{\xi_{j}}|a_{ij}|]\le -\alpha<0$, then
$\frac{d}{dt}\|w(t)\|_{\{\xi, 1\}}\le -\alpha\|w(t)\|_{\{\xi,
1\}}$, which implies $\|w(t)\|_{\{\xi, 1\}}=O(e^{-\alpha t})$.

It happens that these three equalities are closely relating to the matrix measure of $A$ with respect to three norms.
\end{remark}

\subsection{Criteria with $\{\xi,2\}$-norm}
\begin{theorem}\label{th2}
For dynamical systems (\ref{realf}) and (\ref{imagf}), suppose the activation function $f_j(z)$ belongs to class $H_1(\lambda_j^{RR}, \lambda_j^{RI}, \lambda_j^{IR}, \lambda_j^{II})$ and $g_j(z)$ belongs to class $H_2(\mu_j^{RR}, \mu_j^{RI}, \mu_j^{IR}, \mu_j^{II})$, $j=1,\cdots,n$. If there exists a positive vector $\xi=(\xi_1,\cdots,\xi_{n},\phi_1,\cdots,\phi_n)^T>0$ and $\epsilon>0$, such that, for $j=1,\cdots,n$,
\begin{align*}
%\label{t13}
&T13(j)=2\xi_j(-d_j+\epsilon+\{a_{jj}^R\}^{+}\lambda_j^{RR}
+\{-a_{jj}^I\}^{+}\lambda_j^{IR})\nonumber\\
&+\sum_{k=1, k\ne j}^n\xi_j(|a_{jk}^R|\lambda_k^{RR}+|a_{jk}^I|\lambda_k^{IR})
\pi1_{jk}\nonumber\\
&+\sum_{k=1, k\ne j}^n\xi_k(|a_{kj}^R|\lambda_j^{RR}+|a_{kj}^I|\lambda_j^{IR})
\pi1_{kj}^{-1}\nonumber\\
&+\sum_{k=1}^n\xi_j(|a_{jk}^R|\lambda_k^{RI}+|a_{jk}^I|\lambda_k^{II}) \pi2_{jk}\nonumber\\
&+\sum_{k=1}^n\xi_j(|b_{jk}^R|\mu_k^{RR}+|b_{jk}^I|\mu_k^{IR}) \pi3_{jk}\nonumber\\
&+\sum_{k=1}^n\xi_j(|b_{jk}^R|\mu_k^{RI}+|b_{jk}^I|\mu_k^{II})\pi4_{jk}
\nonumber\\
&+\sum_{k=1}^n\phi_k(|a_{kj}^R|
\mu_j^{IR}+|a_{kj}^I|\mu_j^{RR})\omega1^{-1}_{kj}\nonumber\\
&+\bigg(\sum_{k=1}^n\xi_k(|b_{kj}^R|\mu_j^{RR}+|b_{kj}^I|\mu_j^{IR}) \pi3_{kj}^{-1}
\nonumber\\
&+\sum_{k=1}^n\phi_k(|b_{kj}^R|\mu_{j}^{IR}+|b_{kj}^I|\mu_j^{RR}\bigg)\omega3^{-1}_{kj})e^{2\epsilon \tau_{kj}}\le 0,
%\end{align}
%\begin{align}\label{t14}
\\&T14(j)=2\phi_{j}(-d_j+\epsilon+\{a_{jj}^R\}^{+}\mu_j^{RR}+\{a_{jj}^I\}^{+}\mu_j^{RI})\nonumber\\
&+\sum_{k=1}^n\phi_j(|a_{jk}^R|
\mu_k^{IR}+|a_{jk}^I|\mu_k^{RR})\omega1_{jk}\nonumber\\
&+\sum_{k=1, k\ne j}^n\phi_j(|a_{jk}^R|\mu_k^{II}+|a_{jk}^I|\mu_k^{RI}) \omega2_{jk}\nonumber\\
&+\sum_{k=1, k\ne j}^n\phi_k(|a_{kj}^R|\mu_j^{II}+|a_{kj}^I|\mu_j^{RI})\omega2^{-1}_{kj}
\nonumber\\
&+\sum_{k=1}^n\phi_j(|b_{jk}^R|\mu_{k}^{IR}+|b_{jk}^I|\mu_k^{RR})\omega3_{jk}\nonumber\\
&+\sum_{k=1}^n\phi_j(|b_{jk}^R|\mu_k^{II}+|b_{jk}^I|\mu_k^{RI}) \omega4_{jk}\nonumber\\
&+\sum_{k=1}^n\xi_k(|a_{kj}^R|\lambda_j^{RI}+|a_{kj}^I|\lambda_j^{II}) \pi2_{kj}^{-1}\nonumber\\
&+\bigg(\sum_{k=1}^n\xi_k(|b_{kj}^R|\mu_j^{RI}+|b_{kj}^I|\mu_j^{II})\pi4^{-1}_{kj} \nonumber\\
&+\sum_{k=1}^n\phi_k(|b_{kj}^R|\mu_j^{II}+|b_{kj}^I|\mu_j^{RI}) \omega4^{-1}_{kj}\bigg)e^{2\epsilon \tau_{kj}}\le 0,
\end{align*}
where $\pi1_{jk}, \pi2_{jk}, \pi3_{jk}, \pi4_{jk}, \omega1_{jk}, \omega2_{jk}, \omega3_{jk}, \omega4_{jk}$ are positive numbers. Then dynamical systems (\ref{realf}) and (\ref{imagf}) have a unique equilibrium $\overline{Z}^R$ and $\overline{Z}^I$ respectively. Moreover, for any solution $Z(t)$ defined by (\ref{z}), equations (\ref{e1}) and (\ref{e2}) hold, where the norm is $\{\xi, 2\}$-norm.
\end{theorem}

Its proof can be found in Appendix C.

\begin{corollary}\label{cor5}
For dynamical systems (\ref{realf}) and (\ref{imagf}), suppose the activation function $f_j(z)$ belongs to class $H_1(\lambda_j^{RR}, \lambda_j^{RI}, \lambda_j^{IR}, \lambda_j^{II})$ and $g_j(z)$ belongs to class $H_2(\mu_j^{RR}, \mu_j^{RI}, \mu_j^{IR}, \mu_j^{II})$, $j=1,\cdots,n$. If there exists a positive vector $\xi=(\xi_1,\cdots,\xi_{n},\phi_1,\cdots,\phi_n)^T>0$, such that, for $j=1,\cdots,n$,
\begin{align*}
%\label{t15}
T15(j)&=2\xi_j(-d_j+\{a_{jj}^R\}^{+}\lambda_j^{RR}
+\{-a_{jj}^I\}^{+}\lambda_j^{IR})\nonumber\\
&+\sum_{k=1, k\ne j}^n\xi_j(|a_{jk}^R|\lambda_k^{RR}+|a_{jk}^I|\lambda_k^{IR})
\pi1_{jk}\nonumber\\
&+\sum_{k=1, k\ne j}^n\xi_k(|a_{kj}^R|\lambda_j^{RR}+|a_{kj}^I|\lambda_j^{IR})
\pi1_{kj}^{-1}\nonumber\\
&+\sum_{k=1}^n\xi_j(|a_{jk}^R|\lambda_k^{RI}+|a_{jk}^I|\lambda_k^{II}) \pi2_{jk}\nonumber\\
&+\sum_{k=1}^n\xi_j(|b_{jk}^R|\mu_k^{RR}+|b_{jk}^I|\mu_k^{IR}) \pi3_{jk}\nonumber\\
&+\sum_{k=1}^n\xi_j(|b_{jk}^R|\mu_k^{RI}+|b_{jk}^I|\mu_k^{II})\pi4_{jk}
\nonumber\\
&+\sum_{k=1}^n\phi_k(|a_{kj}^R|
\mu_j^{IR}+|a_{kj}^I|\mu_j^{RR})\omega1^{-1}_{kj}\nonumber\\
&+\sum_{k=1}^n\xi_k(|b_{kj}^R|\mu_j^{RR}+|b_{kj}^I|\mu_j^{IR}) \pi3_{kj}^{-1}
\nonumber\\
&+\sum_{k=1}^n\phi_k(|b_{kj}^R|\mu_{j}^{IR}+|b_{kj}^I|\mu_j^{RR})\omega3^{-1}_{kj}<0,
%\end{align}
%\begin{align}\label{t16}
\\T16(j)&=2\phi_{j}(-d_j+\{a_{jj}^R\}^{+}\mu_j^{RR}+\{a_{jj}^I\}^{+}\mu_j^{RI})\nonumber\\
&+\sum_{k=1}^n\phi_j(|a_{jk}^R|
\mu_k^{IR}+|a_{jk}^I|\mu_k^{RR})\omega1_{jk}\nonumber\\
&+\sum_{k=1, k\ne j}^n\phi_j(|a_{jk}^R|\mu_k^{II}+|a_{jk}^I|\mu_k^{RI}) \omega2_{jk}\nonumber\\
&+\sum_{k=1, k\ne j}^n\phi_k(|a_{kj}^R|\mu_j^{II}+|a_{kj}^I|\mu_j^{RI})\omega2^{-1}_{kj}
\nonumber\\
&+\sum_{k=1}^n\phi_j(|b_{jk}^R|\mu_{k}^{IR}+|b_{jk}^I|\mu_k^{RR})\omega3_{jk}\nonumber\\
&+\sum_{k=1}^n\phi_j(|b_{jk}^R|\mu_k^{II}+|b_{jk}^I|\mu_k^{RI}) \omega4_{jk}\nonumber\\
&+\sum_{k=1}^n\xi_k(|a_{kj}^R|\lambda_j^{RI}+|a_{kj}^I|\lambda_j^{II}) \pi2_{kj}^{-1}\nonumber\\
&+\sum_{k=1}^n\xi_k(|b_{kj}^R|\mu_j^{RI}+|b_{kj}^I|\mu_j^{II})\pi4^{-1}_{kj} \nonumber\\
&+\sum_{k=1}^n\phi_k(|b_{kj}^R|\mu_j^{II}+|b_{kj}^I|\mu_j^{RI}) \omega4^{-1}_{kj}< 0,
\end{align*}
where $\pi1_{jk}, \pi2_{jk}, \pi3_{jk}, \pi4_{jk}, \omega1_{jk}, \omega2_{jk}, \omega3_{jk}, \omega4_{jk}$ are positive numbers. Then any solution of systems (\ref{realf}) and (\ref{imagf}) respectively converges to a unique equilibrium
exponentially.
\end{corollary}

\begin{corollary}\label{cor6}
For dynamical systems (\ref{realf}) and (\ref{imagf}), suppose the activation function $f_j(z)$ belongs to class $H_2(\lambda_j^{RR}, \lambda_j^{RI}, \lambda_j^{IR}, \lambda_j^{II})$ and $g_j(z)$ belongs to class $H_2(\mu_j^{RR}, \mu_j^{RI}, \mu_j^{IR}, \mu_j^{II})$, $j=1,\cdots,n$. If there exists a positive vector $\xi=(\xi_1,\cdots,\xi_{n},\phi_1,\cdots,\phi_n)^T>0$, such that, for $j=1,\cdots,n$,
\begin{align*}
&T17(j)=-2\xi_jd_j\nonumber\\
&+\sum_{k=1}^n\xi_j(|a_{jk}^R|\lambda_k^{RR}+|a_{jk}^I|\lambda_k^{IR})
+\sum_{k=1}^n\xi_k(|a_{kj}^R|\lambda_j^{RR}+|a_{kj}^I|\lambda_j^{IR})
\nonumber\\
&+\sum_{k=1}^n\xi_j(|a_{jk}^R|\lambda_k^{RI}+|a_{jk}^I|\lambda_k^{II}) +\sum_{k=1}^n\xi_j(|b_{jk}^R|\mu_k^{RR}+|b_{jk}^I|\mu_k^{IR}) \nonumber\\
&+\sum_{k=1}^n\xi_j(|b_{jk}^R|\mu_k^{RI}+|b_{jk}^I|\mu_k^{II})
+\sum_{k=1}^n\phi_k(|a_{kj}^R|
\mu_j^{IR}+|a_{kj}^I|\mu_j^{RR})\nonumber\\
&+\sum_{k=1}^n\xi_k(|b_{kj}^R|\mu_j^{RR}+|b_{kj}^I|\mu_j^{IR})
+\sum_{k=1}^n\phi_k(|b_{kj}^R|\mu_{j}^{IR}+|b_{kj}^I|\mu_j^{RR})\nonumber\\
&<0,
\\&T18(j)=-2\phi_{j}d_j\nonumber\\
&+\sum_{k=1}^n\phi_j(|a_{jk}^R|
\mu_k^{IR}+|a_{jk}^I|\mu_k^{RR})+\sum_{k=1}^n\phi_j(|a_{jk}^R|\mu_k^{II}+|a_{jk}^I|\mu_k^{RI}) \nonumber\\
&+\sum_{k=1}^n\phi_k(|a_{kj}^R|\mu_j^{II}+|a_{kj}^I|\mu_j^{RI})
+\sum_{k=1}^n\phi_j(|b_{jk}^R|\mu_{k}^{IR}+|b_{jk}^I|\mu_k^{RR})\nonumber\\
&+\sum_{k=1}^n\phi_j(|b_{jk}^R|\mu_k^{II}+|b_{jk}^I|\mu_k^{RI}) +\sum_{k=1}^n\xi_k(|a_{kj}^R|\lambda_j^{RI}+|a_{kj}^I|\lambda_j^{II}) \nonumber\\
&+\sum_{k=1}^n\xi_k(|b_{kj}^R|\mu_j^{RI}+|b_{kj}^I|\mu_j^{II}) +\sum_{k=1}^n\phi_k(|b_{kj}^R|\mu_j^{II}+|b_{kj}^I|\mu_j^{RI})\nonumber\\
&< 0,
\end{align*}
then any solution of systems (\ref{realf}) and (\ref{imagf}) respectively converges to a unique equilibrium
exponentially.
\end{corollary}

\begin{remark}
As for how to use the norms $||\cdot||_{\{\xi,1\}}$ and $||\cdot||_{\{\xi,2\}}$ to discuss the time-varying delayed networks, readers can refer to the papers \cite{CR2004,CW2007a}.
\end{remark}

\section{Numerical example}\label{numerical}
In this section, some numerical simulations are presented to show the effectiveness of our obtained results.

Consider a two-neuron complex-valued recurrent neural network described as follows:
\begin{align}\label{equ1}
\left\{
\begin{array}{ll}
\dot{z}_1(t)=&-dz_1(t)+a_{11}f_1(z_1(t))+a_{12}f_2(z_2(t))\\
&+b_{11}g_1(z_1(t-1))+b_{12}g_2(z_2(t-2))+u_1\\
\dot{z}_2(t)=&-dz_2(t)+a_{21}f_1(z_1(t))+a_{22}f_2(z_2(t))\\
&+b_{21}g_1(z_1(t-3))+b_{22}g_2(z_2(t-4))+u_2
\end{array}
\right.
\end{align}
where $z_k=z^R_k+iz^I_k, k=1,2$, $D=\mathrm{diag}(d,d)=19I_2$, and
\begin{align*}
A=(a_{jk})_{2\times 2}=\left(\begin{array}{cc}-2-3i&3-i\\4-2i&-1+2i\end{array}\right),\\
B=(b_{jk})_{2\times 2}=\left(\begin{array}{cc}-1+2i&2+i\\3-4i&-3+2i\end{array}\right),\\ u=(u_1,u_2)^T=(-3+i,2+4i)^T,\\
f_k(z_k)=\frac{1-\mathrm{exp}(-2z_k^R-z_k^I)}{1+\mathrm{exp}(-2z_k^R-z_k^I)}+i\frac{1}{1+\mathrm{exp}(-z_k^R-2z_k^I)},\\
g_k(z_k)=\frac{1}{1+\mathrm{exp}(-z_k^R-2z_k^I)}+i\frac{1-\mathrm{exp}(-2z_k^R-z_k^I)}{1+\mathrm{exp}(-2z_k^R-z_k^I)}.
\end{align*}

From simple calculations, we have, for $j=1,2$,
\begin{align*}
&0<\frac{\partial{f}_j^R}{\partial{z}^R_j}\le 1=\lambda_{j}^{RR};&0<\frac{\partial{f}_j^R}{\partial{z}^I_j}\le 0.5=\lambda_{j}^{RI};\\
&0<\frac{\partial{f}_j^I}{\partial{z}^R_j}\le 0.25=\lambda_{j}^{IR};&0<\frac{\partial{f}_j^I}{\partial{z}^I_j}\le 0.5=\lambda_{j}^{II};
\end{align*}
therefore, $f_j(z)$ belongs to class $H_1(1, 0.5, 0.25, 0.5), j=1,2$. Similarly, we can prove that $g_j(z)$ belongs to class $H_2(0.25, 0.5, 1, 0.5)$, $j=1,2$.

From the notations defined in (\ref{notg}), we have $\overline{D}=19I_4$,
\begin{align*}
\overline{A}=\left(\begin{array}{cccc}
2&3&3&1\\
4&1&2&2\\
3&1&2&3\\
2&2&4&1
\end{array}\right),
\overline{F}=\left(\begin{array}{cccc}
1&0&0.5&0\\
0&1&0&0.5\\
0.25&0&0.5&0\\
0&0.25&0&0.5
\end{array}\right),\\
\overline{B}=\left(\begin{array}{cccc}
1&2&2&1\\
3&3&4&2\\
2&1&1&2\\
4&2&3&3
\end{array}\right),
\overline{G}=\left(\begin{array}{cccc}
0.25&0&0.5&0\\
0&0.25&0&0.5\\
1&0&0.5&0\\
0&1&0&0.5\end{array}\right).
\end{align*}

Calculations show that eigenvalues of $\overline{D}-\overline{A}\overline{F}-\overline{B}\overline{G}$
are: $-0.7655, 18.6670, 20.9701, 19.8784$, so it is not an M-matrix, which means that Theorem \ref{matrix} is not satisfied.
However, according to Corollary \ref{cor1} and Remark \ref{key}, we have
$P_1=\mathrm{diag}\{2,1\}, P_2=\mathrm{diag}\{0,2\}, P_3=\mathrm{diag}\{3,0\}$, and
\begin{align*}
\overline{\Delta}=\mathrm{diag}\{P_1+0.25P_2,0.5(P_1+P_3)\},
\end{align*}
then eigenvalues of
$\overline{D}-\overline{A}\overline{F}-\overline{B}\overline{G}+\overline{\Delta}$ are $0.8488, 20.0717, 22.7947, 21.5348$, therefore Corollary \ref{cor1} holds, so the above system can achieve its equilibrium exponentially.

The following simulations present the correctness of our claim. We choose five cases for initial values. Case 1: $z_1(t)=-4+3i, z_2(t)=-5-i, t\in [-4, 0]$; Case 2: $z_1(t)=2+i, z_2(t)=-3+2.5i, t\in [-4, 0]$; Case 3: $z_1(t)=3-5i, z_2(t)=6+3i, t\in [-4, 0]$; Case 4: $z_1(t)=-2-4i, z_2(t)=-7+4i, t\in [-4, 0]$; Case 5: $z_1(t)=1+4i, z_2(t)=-5-1.5i, t\in [-4, 0]$. Figures \ref{a}-\ref{d} depict the trajectories of $z_1^R(t), z_1^I(t), z_2^R(t), z_2^I(t)$ respectively. For different initial values, they converge to the same equilibrium $(-0.0351, 0.1423, 0.0912, 0.2239)^T$, i.e., the unique equilibrium has the global exponential stability property.

\begin{figure}
\begin{center}
\includegraphics[width=0.5\textwidth,height=0.27\textheight]{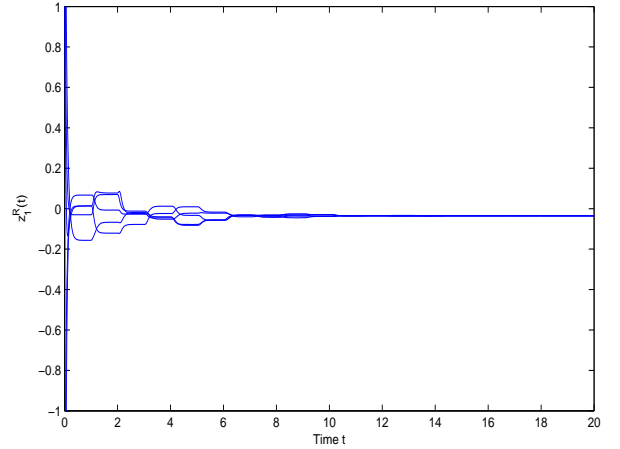}
\end{center}
\caption{Trajectories of $z_1^R(t)$ for different initial values, which show the global exponential stability of equilibrium} \label{a}
\end{figure}
\begin{figure}
\begin{center}
\includegraphics[width=0.5\textwidth,height=0.27\textheight]{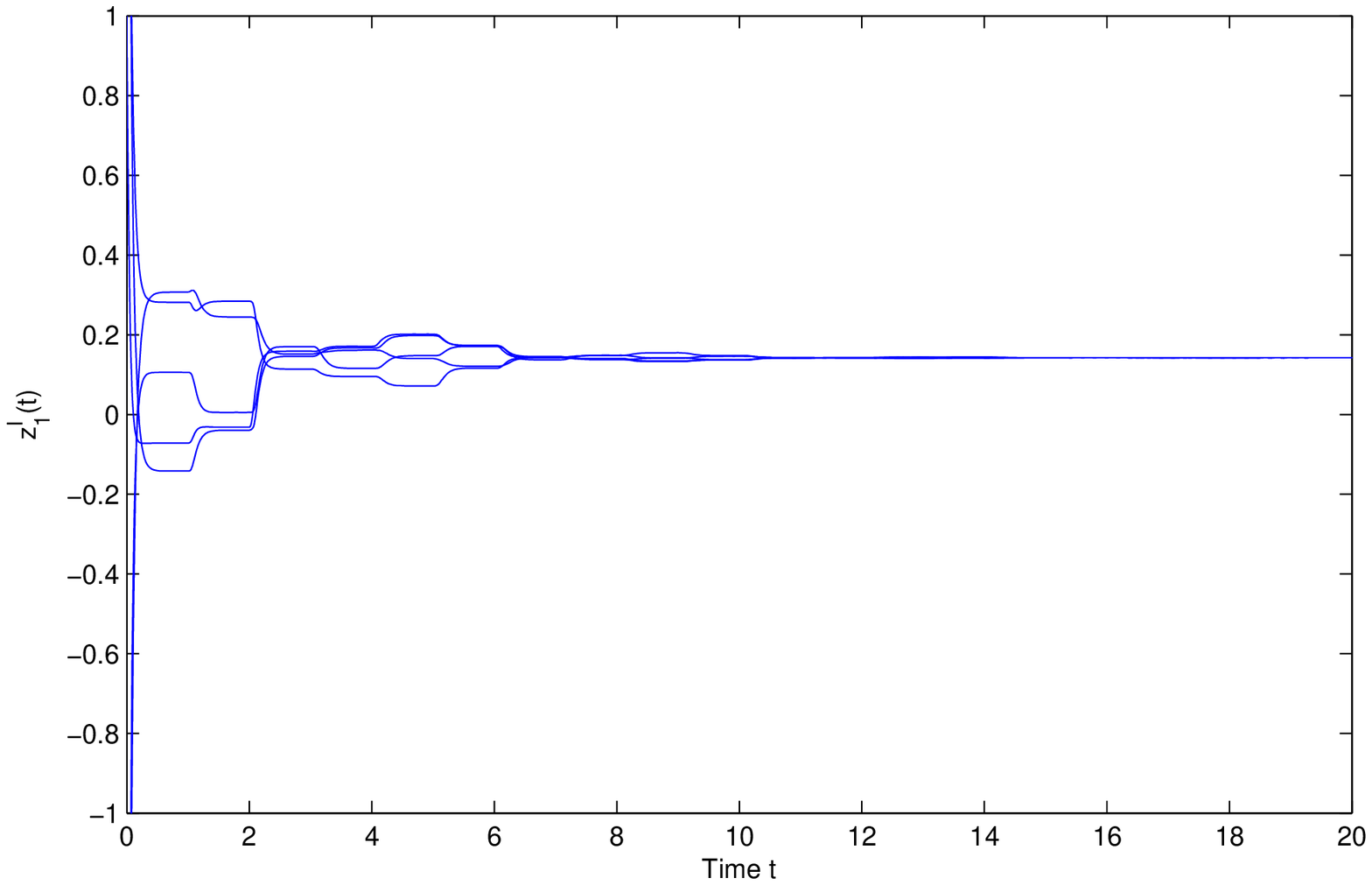}
\end{center}
\caption{Trajectories of $z_1^I(t)$ for different initial values, which show the global exponential stability of equilibrium} \label{b}
\end{figure}
\begin{figure}
\begin{center}
\includegraphics[width=0.5\textwidth,height=0.27\textheight]{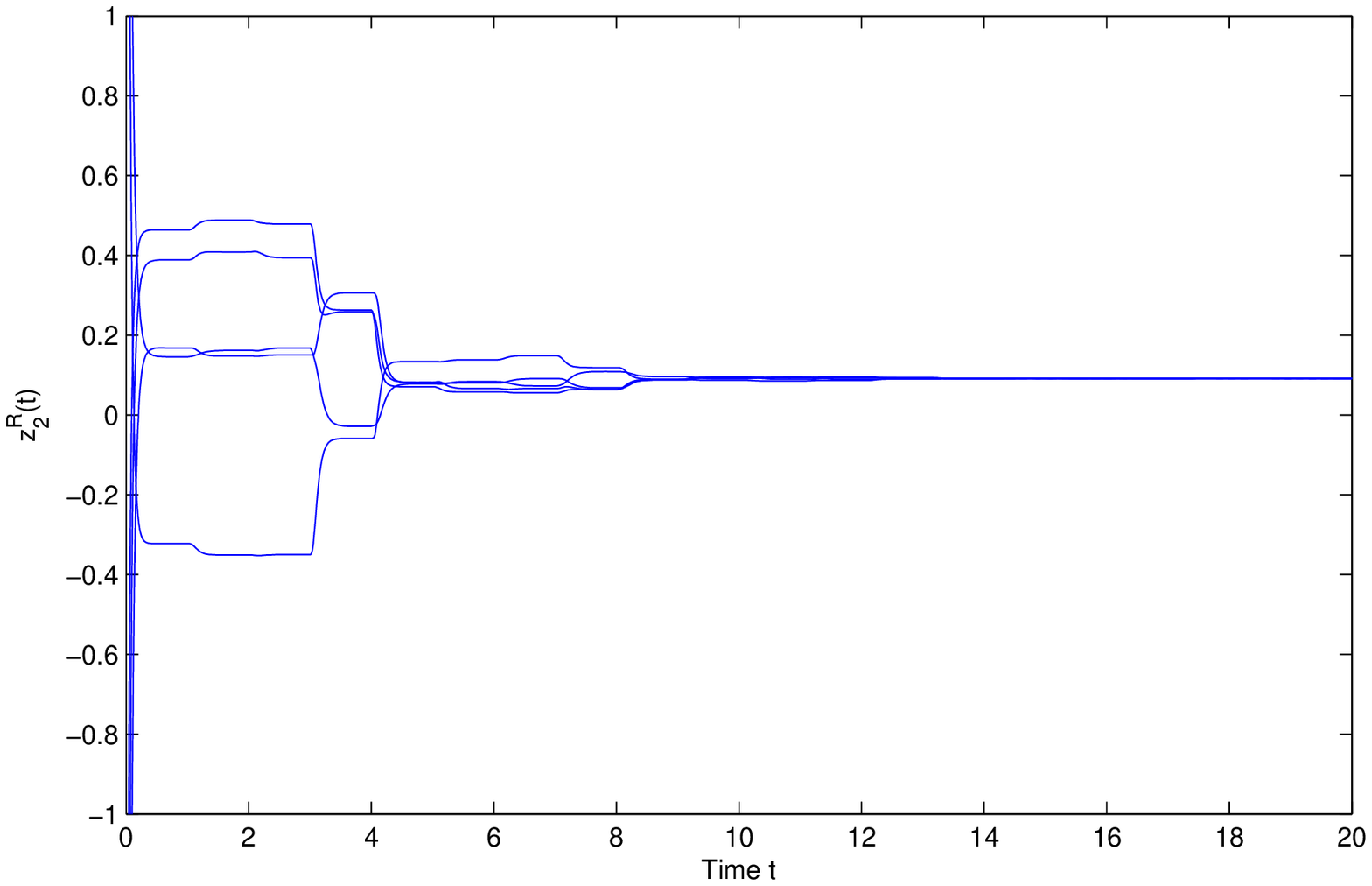}
\end{center}
\caption{Trajectories of $z_2^R(t)$ for different initial values, which show the global exponential stability of equilibrium} \label{c}
\end{figure}
\begin{figure}
\begin{center}
\includegraphics[width=0.5\textwidth,height=0.27\textheight]{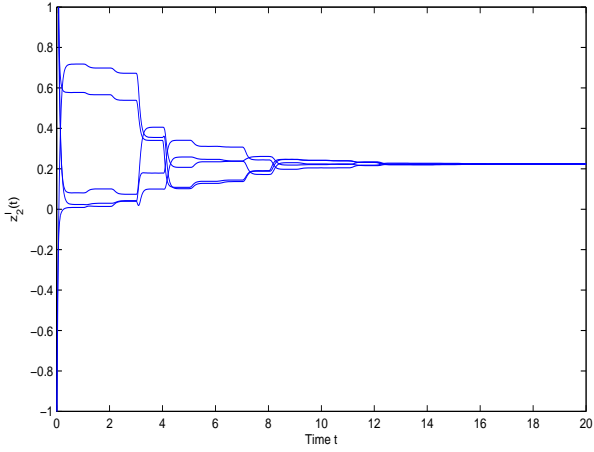}
\end{center}
\caption{Trajectories of $z_2^I(t)$ for different initial values, which show the global exponential stability of equilibrium} \label{d}
\end{figure}

Moreover, if we choose the initial values as Case 1, and only the external control $u$ are different, i.e., different external controls $(u_1, u_2)^T=(-3+i,2+4i)^T$ and $(u_1^{\prime}, u_2^{\prime})^T=(3+2i, 4-i)^T$ are added on the CVNNs, figure \ref{compare} shows that the equilibriums are different, therefore, the equilibrium is heavily impacted by the external control.

\begin{figure}
\begin{center}
\includegraphics[width=0.5\textwidth,height=0.35\textheight]{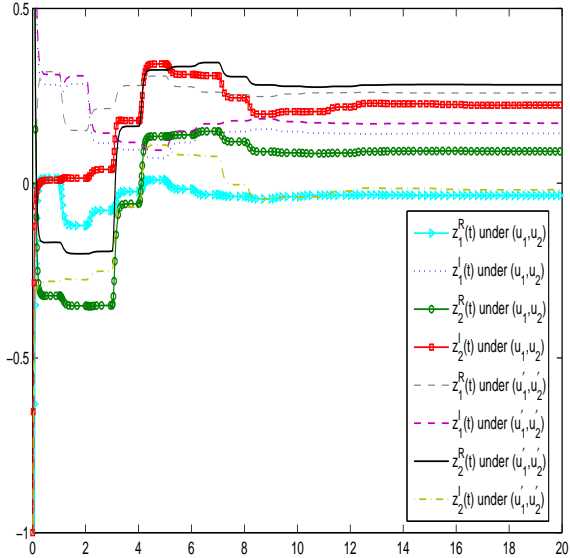}
\end{center}
\caption{Trajectories of all $z_i^R(t)$ and $z_i^I(t), i=1,2$ under different external controllers $(u_1, u_2)$ and $(u_1^{\prime}, u_2^{\prime})$, which means that the equilibrium is impacted by external control } \label{compare}
\end{figure}

In the final simulation, we will consider the time-varying delays, thus we choose the equation as
\begin{align}\label{equ2}
\left\{
\begin{array}{ll}
\dot{z}_1(t)=&-dz_1(t)+a_{11}f_1(z_1(t))+a_{12}f_2(z_2(t))\\
&+b_{11}g_1(z_1(t-1-\sin(t)))\\
&+b_{12}g_2(z_2(t-2-\cos(t)))+u_1\\
\dot{z}_2(t)=&-dz_2(t)+a_{21}f_1(z_1(t))+a_{22}f_2(z_2(t))\\
&+b_{21}g_1(z_1(t-3+\sin(t)))\\
&+b_{22}g_2(z_2(t-4+\cos(t)))+u_2
\end{array}
\right.
\end{align}

All the parameters, including the external control, are the same as defined in the above simulations. Similarly, according to Corollary \ref{cor1}, Remark \ref{ws} and Remark \ref{key}, this system can achieve its equilibrium exponentially. Figures \ref{av}-\ref{dv} depict the trajectories of $z_1^R(t), z_1^I(t), z_2^R(t), z_2^I(t)$ respectively. Moreover, the equilibrium is also $(-0.0351, 0.1423, 0.0912, 0.2239)^T$, i.e., the equilibriums are the same for system (\ref{equ1}) and system (\ref{equ2}) even though they have different time delays.

\begin{figure}
\begin{center}
\includegraphics[width=0.5\textwidth,height=0.27\textheight]{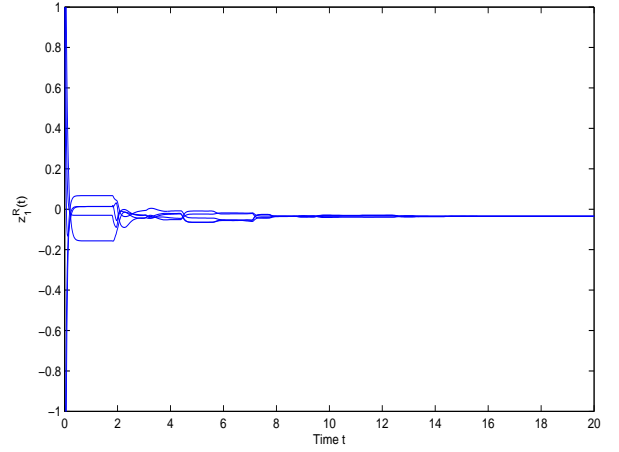}
\end{center}
\caption{Trajectories of $z_1^R(t)$ for different initial values, which show the global exponential stability of equilibrium} \label{av}
\end{figure}
\begin{figure}
\begin{center}
\includegraphics[width=0.5\textwidth,height=0.27\textheight]{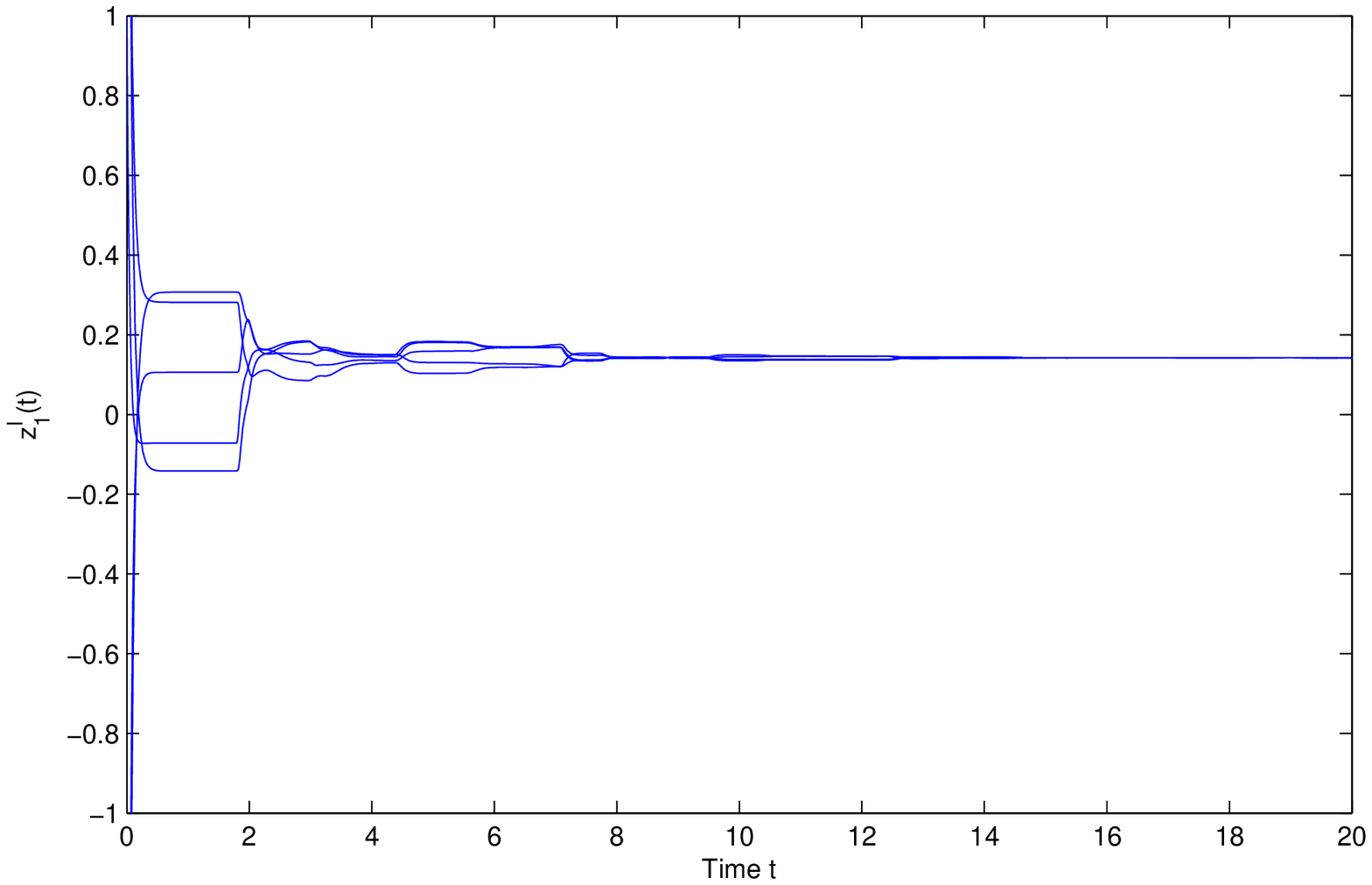}
\end{center}
\caption{Trajectories of $z_1^I(t)$ for different initial values, which show the global exponential stability of equilibrium} \label{bv}
\end{figure}
\begin{figure}
\begin{center}
\includegraphics[width=0.5\textwidth,height=0.27\textheight]{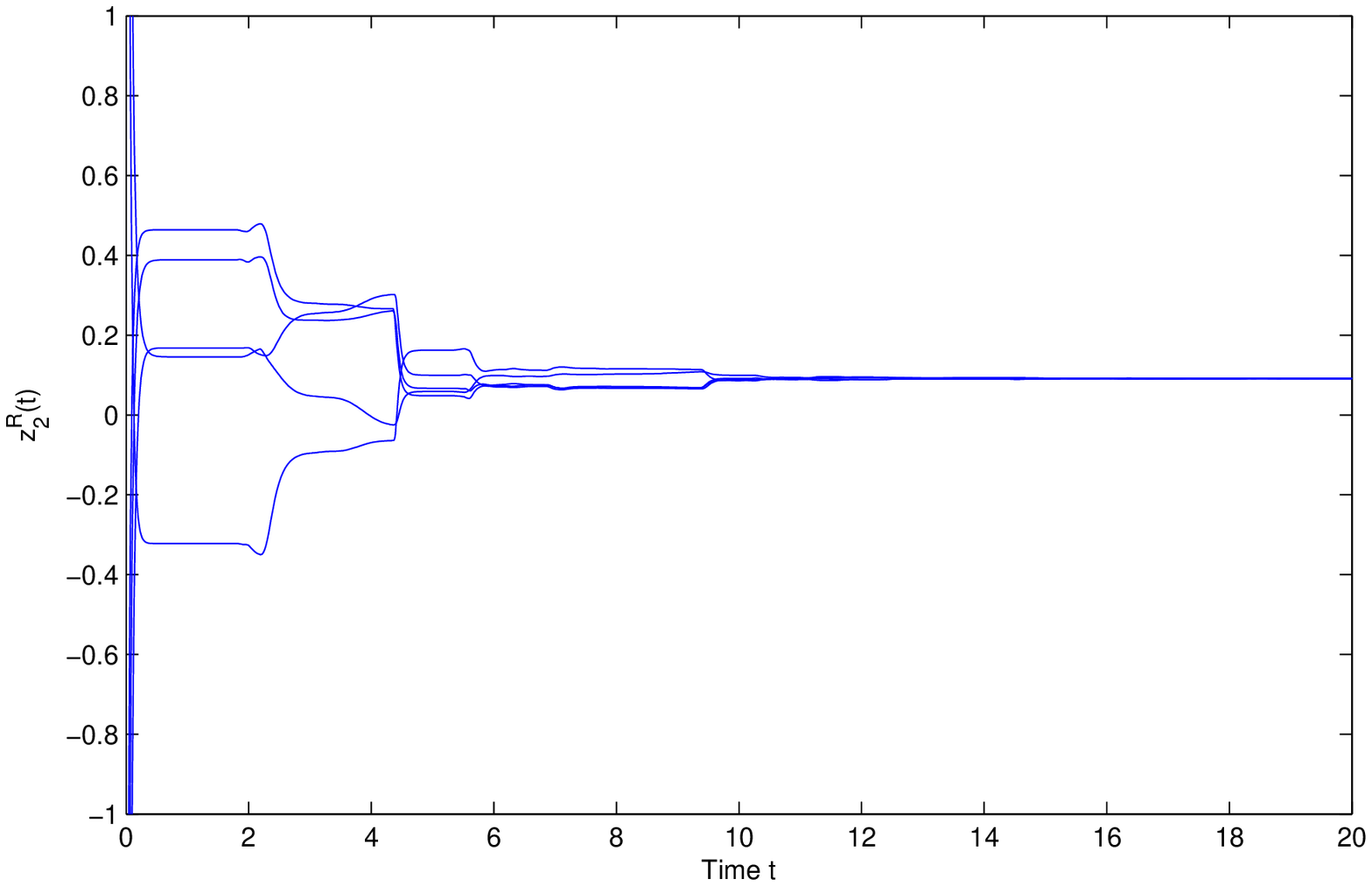}
\end{center}
\caption{Trajectories of $z_2^R(t)$ for different initial values, which show the global exponential stability of equilibrium} \label{cv}
\end{figure}
\begin{figure}
\begin{center}
\includegraphics[width=0.5\textwidth,height=0.27\textheight]{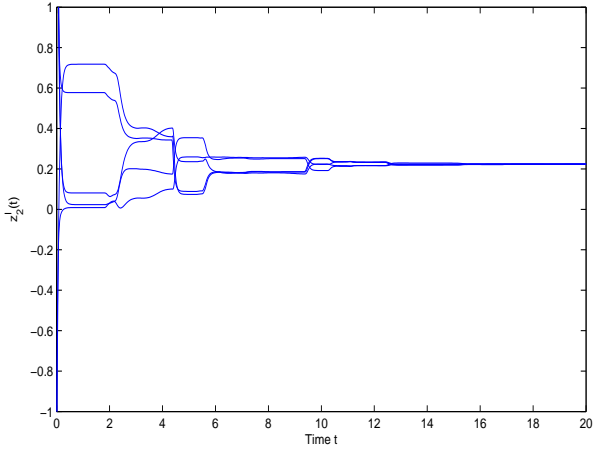}
\end{center}
\caption{Trajectories of $z_2^I(t)$ for different initial values, which show the global exponential stability of equilibrium} \label{dv}
\end{figure}

\section{Conclusion and discussions}\label{conc}
In this paper, we first propose a complex-valued recurrent neural network model with asynchronous time delays. This feature is the first difference of this paper with previous works. Then under the assumptions of activation functions, we prove the exponential convergence directly by using the $\infty$-norm, $1$-norm and $2$-norm respectively, the existence and uniqueness of the equilibrium point is a direct consequence of the exponential convergence; while previous works in the literature always use two proving steps: step 1, prove the existence of equilibrium; step 2, prove the stability. This is also a novelty of this paper for investigating the equilibrium of CVNNs. Moreover, considering the signs of coupling matrix, some sufficient conditions for the uniqueness and global exponential
stability of the equilibrium point are presented, which are more general and less restrictive than previous works, i.e., the $M$-matrix property of $\overline{D}-\overline{A}\overline{F}-\overline{B}\overline{G}$ is just a special case of criteria for exponential stability. These are our main theoretical results. Finally, three numerical examples are given to show the correctness of our obtained results.

In the end, we give some discussions about future directions of the complex-valued neural networks:
\begin{enumerate}
  \item This paper deals with complex-valued neural network by decomposing it to real and imaginary parts and constructing an equivalent real-valued system. To ensure this decomposition, we assume the partial derivatives of activation functions exist and bounded, see Definition \ref{impro}. How to find an efficient way to analyze complex-valued system using the complex nature of system and consider its properties on complex planes will be our future direction.
  \item In this paper, we consider the asynchronous time delays, which can be regarded as discrete delays. However, a distribution of propagation delays can exist for neural networks due to the multitude of parallel pathways with a variety of axon sizes and lengths. Therefore, continuously distributed delays can be a good choice, so investigation of stability under distributed delays will also be our future direction.
  \item As for the dynamical behaviors of complex-valued neural networks, the global existence and exponential stability is just an aspect, there are also many other interesting dynamical behaviors for future research, for example, the multistability, the robustness of uncertain neural networks, the existence and stability of periodic (or almost periodic) solutions, chaotic behaviors for (delayed) complex-valued neural networks, etc.
\end{enumerate}

\section*{Appendix A: The proof of Theorem \ref{thinfty}}
\begin{proof}
Define
\begin{align}\label{new}
x_j(t)=e^{\epsilon t}\dot{z}_j^R(t),~~~~y_j(t)=e^{\epsilon t}\dot{z}_j^I(t), j=1,2,\cdots,n.
\end{align}
Then we have
\begin{align}\label{f1}
&\dot{x}_j(t)
=(-d_j+\epsilon)x_j(t)\nonumber\\
&+\sum_{k=1}^na_{jk}^R
\bigg[\frac{\partial{f}_k^R}{\partial{z}_k^R} x_k+\frac{\partial{f}_k^R}{\partial{z}_k^I} y_k\bigg]-\sum_{k=1}^na_{jk}^I\bigg[\frac{\partial{f}_k^I}{\partial{z}_k^R} x_k+\frac{\partial{f}_k^I}{\partial{z}_k^I} y_k\bigg]\nonumber\\
&+\sum_{k=1}^nb_{jk}^Re^{\epsilon \tau_{jk}} \bigg[\frac{\partial{g}_k^R}{\partial{z}_k^R(\underline{\tau_{jk}})} x_k(\underline{\tau_{jk}})+\frac{\partial{g}_k^R}{\partial{z}_k^I(\underline{\tau_{jk}})} y_k(\underline{\tau_{jk}})\bigg]\nonumber\\
&-\sum_{k=1}^nb_{jk}^Ie^{\epsilon \tau_{jk}}\bigg[\frac{\partial{g}_k^I}{\partial{z}_k^R(\underline{\tau_{jk}})} x_k(\underline{\tau_{jk}})+\frac{\partial{g}_k^I}{\partial{z}_k^I(\underline{\tau_{jk}})} y_k(\underline{\tau_{jk}})\bigg],
\end{align}
and
\begin{align}\label{f2}
&\dot{y}_j(t)=(-d_j+\epsilon)y_j(t)\nonumber\\
&+\sum_{k=1}^na_{jk}^R
\bigg[\frac{\partial{f}_k^I}{\partial{z}_k^R} x_k+\frac{\partial{f}_k^I}{\partial{z}_k^I} y_k\bigg]
+\sum_{k=1}^na_{jk}^I\bigg[\frac{\partial{f}_k^R}{\partial{z}_k^R} x_k+\frac{\partial{f}_k^R}{\partial{z}_k^I} y_k\bigg]\nonumber\\
&+\sum_{k=1}^nb_{jk}^R e^{\epsilon \tau_{jk}}\bigg[\frac{\partial{g}_k^I}{\partial{z}_k^R(\underline{\tau_{jk}})} x_k(\underline{\tau_{jk}})+\frac{\partial{g}_k^I}{\partial{z}_k^I(\underline{\tau_{jk}})} y_k(\underline{\tau_{jk}})\bigg]\nonumber\\
&+\sum_{k=1}^nb_{jk}^Ie^{\epsilon \tau_{jk}}\bigg[\frac{\partial{g}_k^R}{\partial{z}_k^R(\underline{\tau_{jk}})} x_k(\underline{\tau_{jk}})+\frac{\partial{g}_k^R}{\partial{z}_k^I(\underline{\tau_{jk}})} y_k(\underline{\tau_{jk}})\bigg],
\end{align}
where ${\partial{f}_k^{a}}/{\partial{z}_k^b}$ denotes ${\partial{f}_k^{a}(z^R_k(t),z^I_k(t))}/{\partial{z}_k^{b}}, a,b=R, I$;
$x_k(\underline{\tau_{jk}})=x_k(t-\tau_{jk})$, and $y_k(\underline{\tau_{jk}})=y_k(t-\tau_{jk})$; while ${\partial{f}_k^{a}}/{\partial{z}_k^b(\underline{\tau_{jk}})}$ denotes ${\partial{f}_k^{a}(z^R_k(t-\tau_{jk}),z^I_k(t-\tau_{jk}))}/{\partial{z}_k^{b}(t-\tau_{jk})}, a,b=R,I$.

Let
\begin{align*}
X(t)=(x_1(t),\cdots,x_n(t),y_1(t),\cdots,y_n(t))^T\in R^{2n\times 1},
\end{align*}
so $X(t)=e^{\epsilon t}\dot{Z}(t)$, and $\|X(t)\|_{\{\xi,\infty\}}
=\max\{\max_{j}\{|\xi_{j}^{-1}x_{j}(t)|\}, \max_{j}\{|\phi_{j}^{-1}y_{j}(t)|\}\}$.

{\bf Case 1}: For $X(t)$, if $j_0=j_0(t)$, which depends on $t$, is such an index that $|\xi_{j_0}^{-1}x_{j_0}(t)|=\|X(t)\|_{\{\xi,\infty\}}$, then
\begin{align*}
&\xi_{j_0}\frac{d\|X(t)\|_{\{\xi,\infty\}}}{dt}=\frac{d|x_{j_0}(t)|}{dt}\\
=&\mathrm{sign}\{x_{j_0}(t)\}\bigg\{\xi_{j_0}(-d_{j_0}+\epsilon)\xi_{j_0}^{-1}x_{j_0}(t)\\
+&\sum_{k=1}^n\xi_{k}a_{j_0k}^R\frac{\partial{f}_k^R}{\partial{z}_k^R} \xi_k^{-1} x_k+\sum_{k=1}^n\phi_ka_{j_0k}^R\frac{\partial{f}_k^R}{\partial{z}_k^I}\phi_{k}^{-1} y_k\\
-&\sum_{k=1}^n\xi_{k}a_{j_0k}^I\frac{\partial{f}_k^I}{\partial{z}_k^R} \xi_k^{-1} x_k-\sum_{k=1}^n\phi_{k}a_{j_0k}^I\frac{\partial{f}_k^I}{\partial{z}_k^I} \phi_{k}^{-1} y_k\\
+&\sum\limits_{k=1}^n\xi_{k}b_{j_0k}^R\frac{\partial{g}_k^R}{\partial{z}_k^R(\underline{\tau_{j_0k}})}\cdot e^{\epsilon \tau_{j_0k}}\xi_k^{-1} x_k(\underline{\tau_{j_0k}})\\
+&\sum\limits_{k=1}^n\phi_{k}b_{j_0k}^R\frac{\partial{g}_k^R}{\partial{z}_k^I(\underline{\tau_{j_0k}})}\cdot e^{\epsilon \tau_{j_0k}}\phi_{k}^{-1}y_k(\underline{\tau_{j_0k}})\\
-&\sum\limits_{k=1}^n\xi_{k}b_{j_0k}^I\frac{\partial{g}_k^I}{\partial{z}_k^R(\underline{\tau_{j_0k}})}\cdot e^{\epsilon \tau_{j_0k}}\xi_k^{-1} x_k(\underline{\tau_{j_0k}})\\
-&\sum\limits_{k=1}^n\phi_{k}b_{j_0k}^I\frac{\partial{g}_k^I}{\partial{z}_k^I(\underline{\tau_{j_0k}})}\cdot e^{\epsilon \tau_{j_0k}}\phi_{k}^{-1}y_k(\underline{\tau_{j_0k}})\bigg\}\\
\le&\xi_{j_0}(-d_{j_0}+\epsilon+\{a_{j_0j_0}^R\}^{+}\lambda_{j_0}^{RR}+\{-a_{j_0j_0}^I\}^{+}\lambda_{j_0}^{IR})|\xi_{j_0}^{-1}x_{j_0}(t)|\\
+&\sum_{k=1,k\ne j_0}^n\xi_{k}|a_{j_0k}^R|\lambda_{k}^{RR}\cdot \|X(t)\|_{\{\xi,\infty\}}\\
+&\sum_{k=1}^n\phi_{k}|a_{j_0k}^R|\lambda_k^{RI}\cdot\|X(t)\|_{\{\xi,\infty\}}\\
+&\sum_{k=1,k\ne j_0}^n\xi_{k}|a_{j_0k}^I|\lambda_k^{IR}\cdot\|X(t)\|_{\{\xi,\infty\}}\\
+&\sum_{k=1}^n\phi_{k}|a_{j_0k}^I|\lambda_k^{II}\cdot\|X(t)\|_{\{\xi,\infty\}}\\
+&\sum\limits_{k=1}^n\xi_{k}|b_{j_0k}^R|\mu_k^{RR}\cdot e^{\epsilon \tau_{j_0k}}\|X(t-\tau_{j_0k})\|_{\{\xi,\infty\}}\\
+&\sum\limits_{k=1}^n\phi_{k}|b_{j_0k}^R|\mu_k^{RI}\cdot e^{\epsilon \tau_{j_0k}}\|X(t-\tau_{j_0k})\|_{\{\xi,\infty\}}\\
+&\sum\limits_{k=1}^n\xi_{k}|b_{j_0k}^I|\mu_k^{IR}\cdot e^{\epsilon \tau_{j_0k}}\|X(t-\tau_{j_0k})\|_{\{\xi,\infty\}}\\
+&\sum\limits_{k=1}^n\phi_{k}|b_{j_0k}^I|\mu_k^{II}\cdot e^{\epsilon \tau_{j_0k}}\|X(t-\tau_{j_0k})\|_{\{\xi,\infty\}}\\
=&\bigg\{\xi_{j_0}\bigg(-d_{j_0}+\epsilon+\{a_{j_0j_0}^R\}^{+}\lambda_{j_0}^{RR}+\{-a_{j_0j_0}^I\}^{+}\lambda_{j_0}^{IR}\bigg)\\
+&\sum_{k=1,k\ne j_0}^n\xi_{k}|a_{j_0k}^R|\lambda_{k}^{RR}+\sum_{k=1}^n\phi_{k}|a_{j_0k}^R|\lambda_k^{RI}
\\
+&\sum_{k=1,k\ne j_0}^n\xi_{k}|a_{j_0k}^I|\lambda_k^{IR}
+\sum_{k=1}^n\phi_{k}|a_{j_0k}^I|\lambda_k^{II}\bigg\}\cdot\|X(t)\|_{\{\xi,\infty\}}\\
+&\bigg\{\sum\limits_{k=1}^n\xi_{k}|b_{j_0k}^R|\mu_k^{RR}+\sum\limits_{k=1}^n\phi_{k}|b_{j_0k}^R|\mu_k^{RI}
+\sum\limits_{k=1}^n\xi_{k}|b_{j_0k}^I|\mu_k^{IR}\\
&~~~+\sum\limits_{k=1}^n\phi_{k}|b_{j_0k}^I|\mu_k^{II}\bigg\} e^{\epsilon \tau_{j_0k}}\|X(t-\tau_{j_0k})\|_{\{\xi,\infty\}}.
\end{align*}

Furthermore, define
\begin{eqnarray}\label{two}
M(t)=\sup\limits_{t-\tau\le s\le t}\|X(s)\|_{\{\xi,\infty\}},
\end{eqnarray}
where $\tau=\max_{jk}\tau_{jk}$. Then $\|X(t)\|_{\{\xi,\infty\}}\le M(t)$, and if $\|X(t)\|_{\{\xi,\infty\}}= M(t)$, we have
\begin{align}\label{three}
\xi_{j_0}\frac{d\|X(t)\|_{\{\xi,\infty\}}}{dt}\le T1(j_0)\cdot M(t)
\le 0.
\end{align}

{\bf Case 2:} For $X(t)$, if $j_0=j_0^{\prime}(t)$, which depends on $t$, is such an index that $|\phi_{j_0^{\prime}}^{-1}y_{j_0^{\prime}}(t)|=\|X(t)\|_{\{\xi,\infty\}}$, then
\begin{align*}
&\phi_{j_0^{\prime}}\frac{d\|X(t)\|_{\{\xi,\infty\}}}{dt}=\frac{d|y_{j_0^{\prime}}(t)|}{dt}\\
=&\mathrm{sign}\{y_{j_0^{\prime}}(t)\}\bigg\{\phi_{j_0^{\prime}}(-d_{j_0^{\prime}}+\epsilon)\phi_{j_0^{\prime}}^{-1}y_{j_0^{\prime}}(t)\\
+&\sum_{k=1}^n\xi_{k}a_{j_0^{\prime}k}^R\frac{\partial{f}_k^I}{\partial{z}_k^R} \xi_k^{-1} x_k+\sum_{k=1}^n\phi_{k}a_{j_0^{\prime}k}^R\frac{\partial{f}_k^I}{\partial{z}_k^I}\phi_{k}^{-1} y_k\\
+&\sum_{k=1}^n\xi_{k}a_{j_0^{\prime}k}^I\frac{\partial{f}_k^R}{\partial{z}_k^R} \xi_k^{-1} x_k+\sum_{k=1}^n\phi_{k}a_{j_0^{\prime}k}^I\frac{\partial{f}_k^R}{\partial{z}_k^I} \phi_{k}^{-1} y_k\\
+&\sum\limits_{k=1}^n\xi_{k}b_{j_0^{\prime}k}^R\frac{\partial{g}_k^I}{\partial{z}_k^R(\underline{\tau_{j_0^{\prime}k}})} e^{\epsilon \tau_{j_0^{\prime}k}}\xi_k^{-1} x_k(\underline{\tau_{j_0^{\prime}k}})\\
+&\sum\limits_{k=1}^n\phi_{k}b_{j_0^{\prime}k}^R\frac{\partial{g}_k^I}{\partial{z}_k^I(\underline{\tau_{j_0^{\prime}k}})} e^{\epsilon \tau_{j_0^{\prime}k}}\phi_{k}^{-1}y_k(\underline{\tau_{j_0^{\prime}k}})\\
+&\sum\limits_{k=1}^n\xi_{k}b_{j_0^{\prime}k}^I\frac{\partial{g}_k^R}{\partial{z}_k^R(\underline{\tau_{j_0^{\prime}k}})} e^{\epsilon \tau_{j_0^{\prime}k}}\xi_k^{-1} x_k(\underline{\tau_{j_0^{\prime}k}})\\
+&\sum\limits_{k=1}^n\phi_{k}b_{j_0^{\prime}k}^I\frac{\partial{g}_k^R}{\partial{z}_k^I(\underline{\tau_{j_0^{\prime}k}})} e^{\epsilon \tau_{j_0^{\prime}k}}\phi_{k}^{-1}y_k(\underline{\tau_{j_0^{\prime}k}})\bigg\}\\
\le&\bigg\{\phi_{j_0^{\prime}}\bigg(-d_{j_0^{\prime}}+\epsilon+\{a_{j_0^{\prime}j_0^{\prime}}^R\}^{+}\lambda_{j_0^{\prime}}^{II}+\{a_{j_0^{\prime}j_0^{\prime}}^I\}^{+}\lambda_{j_0^{\prime}}^{RI}\bigg)
\\+&\sum_{k=1}^n\xi_{k}|a_{j_0^{\prime}k}^R|\lambda_{k}^{IR}+\sum\limits_{k=1,k\ne j_0^{\prime}}^n\phi_{k}|a_{j_0^{\prime}k}^R|\lambda_k^{II}
+\sum_{k=1}^n\xi_{k}|a_{j_0^{\prime}k}^I|\lambda_k^{RR}\\
+&\sum\limits_{k=1, k\ne j_0^{\prime} }^n\phi_{k}|a_{j_0^{\prime}k}^I|\lambda_k^{RI}\bigg\}\|X(t)\|_{\{\xi,\infty\}}\\
+&\bigg\{\sum\limits_{k=1}^n\xi_{k}|b_{j_0^{\prime}k}^R|\mu_k^{IR}+\sum\limits_{k=1}^n\phi_{k}|b_{j_0^{\prime}k}^R|\mu_k^{II}
+\sum\limits_{k=1}^n\xi_{k}|b_{j_0^{\prime}k}^I|\mu_k^{RR}\\
+&\sum\limits_{k=1}^n\phi_{k}|b_{j_0^{\prime}k}^I|\mu_k^{RI}\bigg\}
e^{\epsilon \tau_{j_0^{\prime}k}}\|X(t-\tau_{j_0^{\prime}k})\|_{\{\xi,\infty\}}.
\end{align*}

From the definition of (\ref{two}), we have $\|X(t)\|_{\{\xi,\infty\}}\le M(t)$, and if $\|X(t)\|_{\{\xi,\infty\}}= M(t)$,
\begin{align}\label{four}
\phi_{j_0^{\prime}}\frac{d\|X(t)\|_{\{\xi,\infty\}}}{dt}
\le T2(j_0^{\prime})\cdot M(t)\le 0.
\end{align}

Therefore, for the above two cases, according to (\ref{three}) and (\ref{four}), one can get that $M(t)$ decreases monotonely, which implies $\|X(t)\|_{\{\xi,\infty\}}=O(1)$ and
\begin{eqnarray*}
\|\dot{Z}(t)\|_{\{\xi,\infty\}}=O(e^{-\epsilon t}),
\end{eqnarray*}
i.e., $\dot{z}_j^R(t)=O(e^{-\epsilon t})$ and $\dot{z}_j^I(t)=O(e^{-\epsilon t}), j=1,2,\cdots,n$.

Consequently, for any $t_1, t_2\in R, t_1>t_2$, there exists a constant $C>0$, such that
\begin{align*}
&\|Z(t_1)-Z(t_2)\|_{\{\xi,\infty\}}=\|\int_{t_2}^{t_1}\dot{Z}(t)dt\|\le \int_{t_2}^{t_1}\|\dot{Z}(t)\|dt\\
\le &\int_{t_2}^{t_1}Ce^{-\epsilon t}dt=\frac{C}{\epsilon}(e^{-\epsilon t_2}-e^{-\epsilon t_1})\le \frac{C}{\epsilon}e^{-\epsilon t_2}.
\end{align*}

By Cauchy convergence principle, we conclude that
$\lim\limits_{t\rightarrow +\infty}Z(t)=\overline{Z}$, for some
$\overline{Z}=({\overline{Z}^R}^T, {\overline{Z}^I}^T)^T$. It is easy to
get that $\overline{Z}$ is an equilibrium point of the systems
(\ref{realf}) and (\ref{imagf}).

Next, we prove that the equilibrium point is unique. Let
$\overline{Z}$ be any equilibrium point of the systems (\ref{realf}) and
(\ref{imagf}). By the same arguments, we can prove that
\begin{align*}
&\|Z(t)-\overline{Z}\|_{\{\xi,\infty\}}=\|\int_{t}^{\infty}\dot{Z}(t)dt\|\le
\frac{C}{\epsilon}e^{-\epsilon t}.
\end{align*}
which means that any solution $Z(t)$ converges to $\overline{Z}$
exponentially and the equilibrium point is unique.
\end{proof}

\section*{Appendix B: Proof of Theorem \ref{th1}}

\begin{proof}
Recall the definition of $x_j(t)$ and $y_j(t)$ defined in (\ref{new}), we can define a Lyapunov function as
\begin{align*}
L_1(t)=&\sum_{j=1}^n\xi_j|x_j(t)|+\sum_{j=1}^n\phi_{j}|y_j(t)|\\
&+\sum\limits_{j,k=1}^n\alpha_{jk}e^{\epsilon \tau_{jk}}\int_{t-\tau_{jk}}^t|x_k(s)|ds\\
&+\sum\limits_{j,k=1}^n\beta_{jk}e^{\epsilon \tau_{jk}}\int_{t-\tau_{jk}}^t|y_k(s)|ds,
\end{align*}
where
\begin{align*}
\alpha_{jk}&=\xi_j(|b_{jk}^R|\mu_k^{RR}+|b_{jk}^I|\mu_k^{IR})+\phi_{j}(|b_{jk}^R|\mu_k^{IR}+|b_{jk}^I|\mu_k^{RR});\\
\beta_{jk}&=\xi_j(|b_{jk}^R|\mu_k^{RI}+|b_{jk}^I|\mu_k^{II})+\phi_{j}(|b_{jk}^R|\mu_k^{II}+|b_{jk}^I|\mu_k^{RI}).
\end{align*}

Differentiating $L_1(t)$ along equations (\ref{f1}) and (\ref{f2}), using some calculations (the details are left to interested readers), we have
\begin{align*}
\dot{L}_1(t)\le\sum\limits_{k=1}^nT7(k)\cdot|x_k(t)|+\sum\limits_{k=1}^nT8(k)\cdot|y_k(t)|\le 0.
\end{align*}

By similar arguments used in the proof of Theorem \ref{thinfty}, it is easy to see that the equilibrium point is unique.
\end{proof}

\section*{Appendix C: Proof of Theorem \ref{th2}}
\begin{proof}
Recall the definition of $x_j(t)$ and $y_j(t)$ defined in (\ref{new}), we can define a Lyapunov function as
\begin{align*}
L_2(t)=&\sum\limits_{j=1}^n\xi_jx_j^2(t)+\sum\limits_{j=1}^n\phi_{j}y_j^2(t)\\
&+\sum\limits_{j,k=1}^n\alpha^{\prime}_{jk}e^{2\epsilon \tau_{jk}}\int_{t-\tau_{jk}}^t x_k^2(s)ds\\
&+\sum\limits_{j,k=1}^n\beta^{\prime}_{jk}e^{2\epsilon \tau_{jk}}\int_{t-\tau_{jk}}^t y_k^2(s)ds,
\end{align*}
where
\begin{align*}
\alpha^{\prime}_{jk}=&\sum_{k=1}^n\xi_j(|b_{jk}^R|\mu_k^{RR}+|b_{jk}^I|\mu_k^{IR}) \pi3_{jk}^{-1}\\
&+\sum_{k=1}^n\phi_j(|b_{jk}^R|\mu_{k}^{IR}+|b_{jk}^I|\mu_k^{RR})\omega3^{-1}_{jk};\\
\beta^{\prime}_{jk}=&\sum_{k=1}^n\xi_j(|b_{jk}^R|\mu_k^{RI}+|b_{jk}^I|\mu_k^{II})\pi4^{-1}_{jk} \\&+\sum_{k=1}^n\phi_j(|b_{jk}^R|\mu_k^{II}+|b_{jk}^I|\mu_k^{RI}) \omega4^{-1}_{jk}.
\end{align*}

Differentiating $L_2(t)$ along equations (\ref{f1}) and (\ref{f2}), using some calculations (the details are left to interested readers), one can get that
\begin{align*}
\dot{L}_2(t)\le\sum\limits_{j=1}^nT13(j)x_j^2(t)+\sum\limits_{j=1}^nT14(j)y^2_j(t)\le 0.
\end{align*}

By similar arguments used in the proof of Theorem \ref{thinfty}, it is easy to see that the equilibrium point is unique.
\end{proof}

\end{document}